\documentclass[12pt]{article}

\usepackage{amsmath, amsthm, amssymb,graphicx,wrapfig, url}

\newtheorem{theorem}{\bf Theorem}[section]

\theoremstyle{definition}
\newtheorem{example}{\bf Example}[section]
\newcommand{\vecrho}{{\boldsymbol \rho}}

\RequirePackage[pdfusetitle]{hyperref}
\hypersetup{colorlinks,urlcolor=[rgb]{0,0,1},linkcolor=[rgb]{0,0,0},citecolor=[rgb]{0,0,0}}

\usepackage{mathtools}

\begin{document}

\title{Rigid folding equations of degree-6 origami vertices}

\author{
Johnna Farnham\thanks{Tufts University, Department of Mathematics, Medford, MA}, Thomas C. Hull\thanks{Western New England University, Department of Mathematics, Springfield, MA, {\tt thull@wne.edu}}, Aubrey Rumbolt\thanks{Taconic High School, Department of Mathematics, Pittsfield, MA} }


\maketitle

\begin{abstract}
Rigid origami, with applications ranging from nano-robots to unfolding solar sails in space, describes when a material is folded along straight crease line segments while keeping the regions between the creases planar.  Prior work has found explicit equations for the folding angles of a flat-foldable degree-4 origami vertex and some cases of degree-6 vertices. We extend this work to generalized symmetries of the degree-6 vertex where all sector angles equal $60^\circ$. We enumerate the different viable rigid folding modes of these degree-6 crease patterns and then use $2^{nd}$-order Taylor expansions and prior rigid folding techniques to find algebraic folding angle relationships between the creases.  This allows us to explicitly compute the configuration space of these degree-6 vertices, and in the process we uncover new explanations for the effectiveness of Weierstrass substitutions in modeling rigid origami.  These results expand the toolbox of rigid origami mechanisms that engineers and materials scientists may use in origami-inspired designs.
\end{abstract}


\section{Introduction}

Over the past ten years there has been a surge of interest in applications of origami, the art of paper folding, in engineering \cite{Evans}, physics \cite{Silverberg}, and architecture \cite{Hoberman}.  Of particular use to these fields is \textit{rigid origami}, where we attempt to flex a flat sheet made of stiff polygons joined edge-to-edge by hinges, called the \textit{crease pattern} of the origami \cite{Tachi1}.  Such flexible crease patterns have the advantage of being flat and thus easy to manufacture, and the mechanics present during the folding and unfolding process operates independent of scale, whether used for deploying large solar panels in space \cite{Lang} or microscopic polymer gel membranes \cite{gels}.

An essential ingredient for the design of rigid origami structures is a thorough understanding of the \textit{configuration space} ${\cal C}$ of the rigidly folding crease pattern, a union of manifolds in $\mathbb{R}^n$, with $n$ the number of creases, where each coordinate tracks the \textit{folding angle} (the signed deviation from the unfolded state) of a crease.  Knowing explicit relationships between the folding angles during the rigid folding and unfolding process allows designers 
to calculate relative folding speeds between adjacent panels and plan rotational spring torque to actuate each crease in order to have the structure \textit{self-fold} to a desired state \cite{Tachi2,Stern}.  Normally engineers prefer rigid origami crease patterns that possess a single degree of freedom (DOF), such as crease patterns that consist only of degree-4 vertices, because their configuration spaces are completely understood \cite{Izmestiev}.

In this paper we show how configuration spaces can be similarly well-understood for rigid origami vertices with higher DOF if we impose symmetry on the folding angles.  This approach is of use in applications because actuators for folding creases can often be programmed to impose symmetry on the rigid folding, e.g. by using springs of prescribed strengths.  Additionally, adding symmetry constraints will reduce the DOF of the rigid origami vertex, making the folding mechanics more controllable for design.

The idea of using crease pattern symmetry to reduce DOF and obtain folding angle relationships is not new; see for instance \cite{Tachi2,Chen,Feng,Zhang-Chen,Ma}.  However, in this study we undertake the first complete examination of all symmetry possibilities
of the degree-6 origami vertex whose sector angles between creases are all $60^\circ$.  After establishing background and prior work in Section~\ref{sec2}, we will compute the number of symmetrically-different patterns we can place on the sequence of folding angles $(\rho_1,\ldots, \rho_6)$ that are rigidly foldable in Section~\ref{sec3},  and find explicit equations for the folding angle relationships in each case in Section~\ref{sec4}.   Our methods will also allow us to generalize away from the ``all $60^\circ$ angles" crease pattern, so long as the folding angle symmetry can be preserved.

We will also see that, as in the degree-4 case, the Weierstrass substitution $t_i=\tan(\rho_i/2)$ leads to especially nice folding angle relations in many of our degree-6 cases.  We offer an explanation as to why this is the case in Section~\ref{sec5} and offer a list of open questions for future study.

\section{Background on rigid foldings}\label{sec2}

We follow the notation from \cite{Origametry}.  A crease pattern $(G,P)$ is a straight-line embedding of a planar graph $G=(V,E)$ on a compact region $P\subset\mathbb{R}^2$.  
We define a continuous function $f:P\to\mathbb{R}^3$ to be a \textit{rigid folded state} of the crease pattern $G$ on $P$ if for every face $F$ of $G$ (including the faces that border the boundary of $P$) we have that  $f$ restricted to $F$ is an isometry and our folded image $f(P)$ has no self-intersections.\footnote{Self-intersections in rigid origami are difficult to quantify in general \cite{belcastro}, but since we focus on folding crease patterns with only one vertex, we can treat self-intersections informally as when two faces of $G$ have intersecting interiors in $f(P)$.} Each crease line $c_i\in E$ will border two faces, and the signed supplement of the dihedral angle between these faces in a rigid folded state $f$ is called the \textit{folding angle} $\rho_i$ of $c_i$ under $f$.  We denote the set of folding angles for all the creases in a rigid folded state by the \textit{folding angle function} $\mu(c_i)=\rho_i$.  If $\rho_i=0$ then the crease $c_i$ is \textit{unfolded} and if $\rho_i>0$ (resp. $\rho_i<0$) then $c_i$ is said to be a \textit{valley} (resp. \textit{mountain}) crease. If any crease has $\rho_i=\pm\pi$ then that crease has been \textit{folded flat} and represents the largest magnitude a folding angle can achieve.

If $(G,P)$ has $n$ creases then the \textit{parameter space} of the crease pattern is $[-\pi,\pi]^n$, where each point $\vecrho\in [-\pi,\pi]^n$ has coordinates equal to a possible folding angle $\rho_i$ of crease $c_i$ in a rigid folded state $f$. The \textit{configuration space} ${\cal C}(G)$ of the crease pattern is the subset of points $\vecrho=(\rho_1,\ldots, \rho_n)$ such that there exists a rigid folded state $f$ whose folding angle function is $\mu(c_i)=\rho_i$.  We then say  that a crease pattern $(G,P)$ has a \textit{rigid folding} if there exists a path $\gamma:[a,b]\to{\cal C}(G)$ and a parameterized family of rigid folded states $f(t)$ on $(G,P)$ such that the folding angle function for $f(t)$ is given by $\gamma(t)$. 

One of the main problems in rigid origami theory is to describe the configuration space of a given crease pattern, which can be done by finding expressions for the folding angle functions.  Certainly a first step is to do this for \textit{single-vertex} crease patterns $(G,P)$, where only one vertex of $G$ is in the interior of $P$.  The case where this vertex is degree four and \textit{flat-foldable} (meaning each crease can be folded flat simultaneously) is completely understood and characterized as follows:

\begin{theorem}\label{thm:deg4}
Let $(G,P)$ be a degree-4 flat-foldable vertex whose plane angles around the vertex are (in order) $\alpha, \beta, \pi-\alpha$, and $\pi-\beta$, where $0<\alpha< \beta\leq \pi/2$, arranged among the creases $c_1,\ldots, c_4$ as in Figure~\ref{fig:deg4}.  Let $\mu(c_i)=\rho_i$ be the folding angle function for a rigid folded state $f$ of $(G,P)$.  Then there are only two possibilities for an explicit representation of $\mu$, which we call mode 1 and 2:
\begin{align}
\mbox{Mode 1: } & \rho_1=-\rho_3,\ \rho_2=\rho_4,\mbox{ and }\tan\frac{\rho_1}{2}=p(\alpha,\beta)\tan\frac{\rho_2}{2}\\
\mbox{Mode 2: } & \rho_1=\rho_3,\ \rho_2=-\rho_4,\mbox{ and } \tan\frac{\rho_2}{2} = q(\alpha,\beta)\tan\frac{\rho_1}{2}
\end{align}
where 
$$p(\alpha,\beta)= \frac{ \cos\frac{\alpha+\beta}{2}}{\cos\frac{\alpha-\beta}{2} }=
\frac{1-\tan\frac{\alpha}{2}\tan\frac{\beta}{2}} {1+\tan\frac{\alpha}{2}\tan\frac{\beta}{2}} 
\mbox{ and }
q(\alpha,\beta) =  \frac{ \sin\frac{\alpha-\beta}{2}}{\sin\frac{\alpha+\beta}{2} } = 
\frac{\tan\frac{\alpha}{2}-\tan\frac{\beta}{2}}{\tan\frac{\alpha}{2}+\tan\frac{\beta}{2} }. $$
\end{theorem}

\begin{figure}
    \centering
    \includegraphics[width=\linewidth]{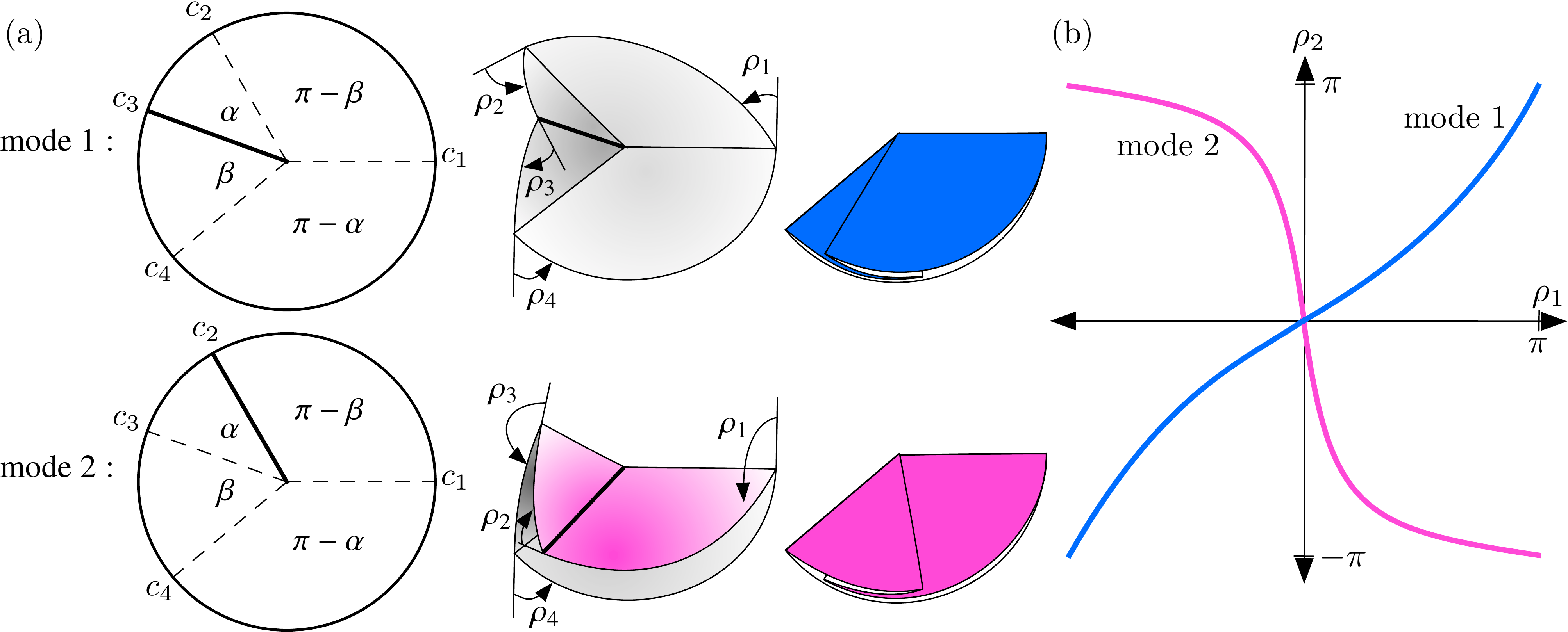}
    \caption{(a) Modes 1 and 2 for a flat-foldable, degree-4 vertex.  (b) The configuration space ${\cal C}(G)$ is the union of two curves, one for each mode. Bold creases are mountains,  others are valleys.}
    \label{fig:deg4}
\end{figure}

See \cite{Tachi2,Origametry} for a proof.  Also, the fact that the opposite plane angles around the vertex are supplementary follows from \textit{Kawasaki's Theorem}, and the fact that we must have one mountain and three valley creases (or vice-versa) as in Figure~\ref{fig:deg4}(a) follows from \textit{Maekawa's Theorem}.  These are two basic theorems on flat-foldable vertices and can be found in \cite{Origametry}.  

The fact that the folding angle expressions in Theorem~\ref{thm:deg4} are linear when parameterized by the Weierstrass substitution is surprising.  We refer to the constants $p(\alpha,\beta)$ and $q(\alpha,\beta)$ as \textit{folding angle multipliers} of their respective rigid foldings.

Also note that the case where $\alpha=\beta<\pi/2$ gives us a crease pattern, called a \textit{bird's foot}, with a line of reflective symmetry.  This results in $p(\alpha,\beta)=q(\alpha,\beta)$, implying that there is only one rigid folding mode.

A tool that can be used to prove Theorem~\ref{thm:deg4}, and that we will employ in this paper, involves rotations about the crease lines to simulate the rigid folding.  Let $R(c_i,\rho_i)\in SO(3)$ be the $3\times 3$ orthogonal matrix that rotates $\mathbb{R}^3$ by $\rho_i$ about the line containing crease $c_i$ in $(G,P)$, where we place the vertex of $G$ at the origin and let $P$ lie in the $xy$-plane.  Define
\begin{equation}\label{eq:R=I}
F(\vecrho)=R(c_1,\rho_1)R(c_2,\rho_2)\cdots R(c_n,\rho_n).
\end{equation}
Then $\vecrho\in{\cal C}(G)$ implies that $F(\vecrho)=I_3$, the $3\times 3$ identity matrix \cite{belcastro}.  It has been proven \cite{Hodge,dual-origami} that a vertex is rigidly-foldable if and only if it is \textit{2nd-order} rigidly-foldable, meaning that the 2nd-order Taylor expansion (about the origin) of $F(\vecrho)$, say where the folding angles are parameterized by time $t$, equals $I_3$.  Writing the folding angles as $\rho_i(t)$, the Taylor expansion is
$$F(\vecrho(t))=  I+\sum_{i=1}^n \left.\frac{\partial F}{\partial\rho_i}\rho'_i(t)\right|_{t=0} t+ \frac{1}{2}
\left.\left(\sum_{i=1}^n\sum_{j=1}^n \frac{\partial^2 F}{\partial\rho_i\partial\rho_j}\rho'_i(t)\rho'_j(t)
+\sum_{i=1}^n\frac{\partial F}{\partial\rho_i}\rho''_i(t)\right)\right|_{t=0} t^2 + o(t^3).$$
Therefore the linear and quadratic terms must be zero. The linear term equalling zero corresponds to the origami vertex being infinitesimally rigidly foldable at the origin (i.e., the unfolded state).  The second-order terms being zero is what guarantees a finite-length rigid folding motion away from the origin.  Following \cite{dual-origami} we can compute these second-order terms from \eqref{eq:R=I} and arrive at the equation
\begin{equation}\label{eq:2nd-order}
\sum_{i,j} 
\begin{pmatrix}
-l_i^yl_j^y & l_a^yl_b^x & 0\\
l_a^xl_b^y & -l_i^x l_j^x & 0\\
0 & 0 &-l_i^x l_j^x - l_i^y l_j^y
\end{pmatrix}
\rho'_i(0)\rho'_j(0) +\sum_i
\begin{pmatrix}
0 & 0 & l_i^y\\
0 & 0 & -l_i^x\\
-l_i^y & l_i^x & 0
\end{pmatrix} \rho''_i(0)=Z,
\end{equation}
where $Z$ is the zero matrix, $a=\min(i,j)$, $b=\max(i,j)$, and we consider our single-vertex creases to be vectors $c_i=(l_i^x, l_i^y, l_i^z)$ (where $l_i^z=0$ at the unfolded state, when $t=0$).  As argued in \cite{dual-origami}, we may assume the folding angle acceleration terms $\rho_i''(0)$ are zero at the origin (i.e., we may assume $\rho_i'$ is constant near the origin because rigid folding paths through the origin will always be mountain-valley symmetric).  Thus for a rigid folding to exist in a neighborhood of the origin, we need to find velocities $\rho_i'(0)$ that make the first matrix sum in Equation~\eqref{eq:2nd-order} the zero matrix.  Surprisingly, often this local approximation around the origin leads to folding angle functions that hold globally for the whole configuration space.

\begin{example}[Trifold]\label{ex:trifold}
As an example, we will use the 2nd-order matrix identity \eqref{eq:2nd-order} to find the folding angle function for $(G_{60},P_{60})$, the degree-6 vertex whose plane angles are all $60^\circ$ (Figure~\ref{fig:trifold}(a)) and that rigidly folds with folding angle symmetry $(\rho_1,\rho_2,\rho_1,\rho_2,\rho_1,\rho_2)$ (Figure~\ref{fig:trifold}(b)) for the creases $c_1,\ldots, c_6$. (This was first derived in \cite{Tachi2,HullTachi1}.) We call this rigid folding symmetry the \textit{trifold} way to fold this vertex. 

\begin{figure}
    \centering
    \includegraphics[width=\linewidth]{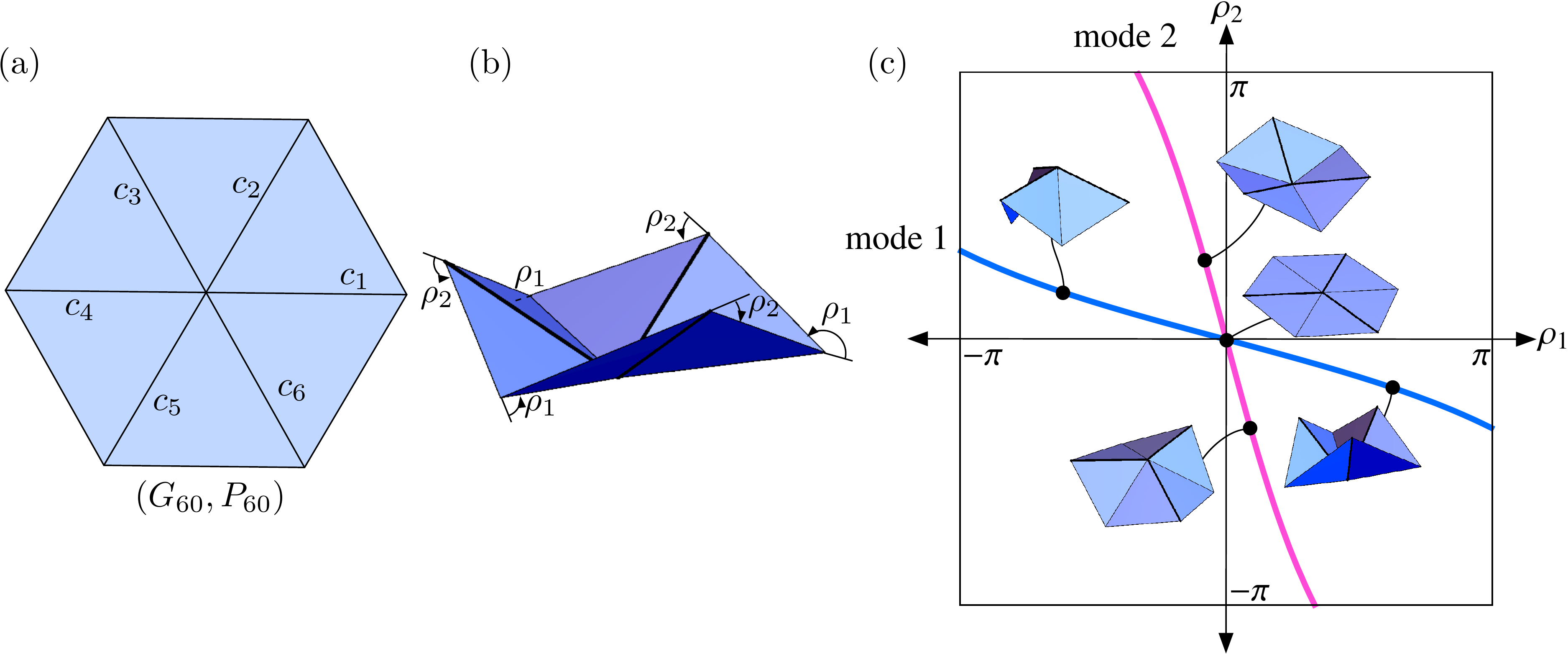}
    \caption{The trifold example. (a) The degree-6, all-$60^\circ$-angle single-vertex crease pattern $(G_{60},P_{60})$. (b) A symmetric $(\rho_1,\rho_2,\rho_1,\rho_2,\rho_1,\rho_2)$ rigid folding of $(G,P)$.  (c) The configuration space curves, which is a subset of the full ${\cal C}(G_{60})$ space, with sample points illustrated by their respective rigid foldings.}
    \label{fig:trifold}
\end{figure}

We set $c_i=(\cos((i-1)\pi/3), \sin((i-1)\pi/3),0)$ for $i=1,\ldots, 6$. Then the first matrix sum in \eqref{eq:2nd-order} becomes $$\begin{pmatrix}
0 & -A & 0\\
A & 0 & 0\\
0 & 0 & 0
\end{pmatrix}\mbox{ where } 
A=\frac{\sqrt{3}}{2}\rho_1'(0)^2+2\sqrt{3}\rho_1'(0)\rho_2'(0)+\frac{\sqrt{3}}{2}\rho_2'(0)^2.$$
Setting this equal to zero and solving for $\rho_2'$ yields two possible folding modes:  $\rho_2'=-(2+\sqrt{3})\rho_1'$ and $\rho_2'=-(2-\sqrt{3})\rho_1'$.  There are many $\rho_1, \rho_2$ relationships that could satisfy these linear differential equations at the origin, such as anything of the form $(\rho_1,\rho_2)=(t+o(t^2), -(2+\sqrt{3})t + o(t^2))$ for the first mode. However, taking a page from the degree-4 case, one should consider the Weierstrass substitution.  It turns out that a \textit{modified} Weierstrass substitution $\tan(\rho_i/4)$ does the trick, and the global folding angle relationships for the trifold are
\begin{equation}\label{eq:trifold}
\tan\frac{\rho_1}{4}=-(2+\sqrt{3})\tan\frac{\rho_2}{4}\mbox{ for mode 1 and }
\tan\frac{\rho_2}{4}=-(2+\sqrt{3})\tan\frac{\rho_1}{4}\mbox{ for mode 2,}
\end{equation}
as proved in \cite{Tachi2,HullTachi1}. These curves are shown in Figure~\ref{fig:trifold}(c).
\end{example}

In general a rigid folding vertex of degree $n$ will have $n-3$ DOF, and so $(G_{60},P_{60})$ will have 3 DOF as it rigidly folds.  Unlike the degree-4 case, there are no known elegant folding angle equations for degree-6 rigid foldings; while such equations can be computed using standard kinematics methods, the resulting equations are quite unwieldy (see Section~\ref{sec4}\ref{sec:other}).   The symmetry imposed in Example~\ref{ex:trifold}, however, reduces the DOF to one, and very nice folding angle equations result.  We can think of the Equations~\eqref{eq:trifold} as carving out a curve for this particular symmetric case from the larger configuration space ${\cal C}(G_{60})$.

\section{Different symmetries of the degree-6 vertex}\label{sec3}

\begin{figure}[b]
    \centering
    \includegraphics[width=\linewidth]{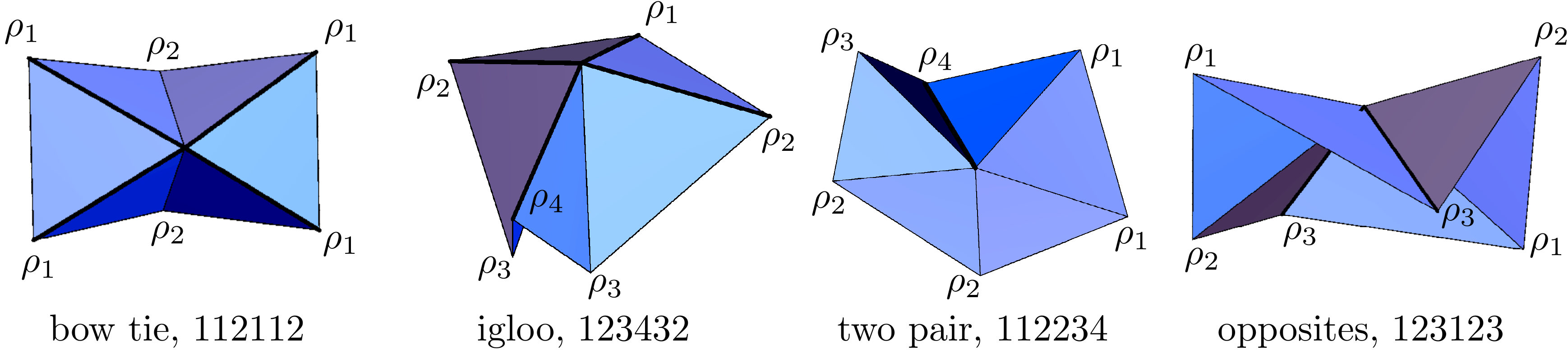}
    \caption{The symmetric $G_{60}$ rigid folding cases from Table~\ref{table1} (excluding the trifold). The number sequence with each case indicates the respective bracelet pattern (equivalent to the indices of the folding angles, in order).}
    \label{fig:symmcases}
\end{figure}

It therefore makes sense to explore all the different symmetry types that the folding angles of $(G_{60},P_{60})$ can achieve and see if their folding angle equations can be found.  A rigid folding of $(G_{60},P_{60})$ with no symmetry would have all different folding angles, $(\rho_1,\rho_2,\rho_3,\rho_4,\rho_5,\rho_6)$.  If we equate the different $\rho_i$ with colors, then classifying the differently-symmetric rigid foldings of $(G_{60},P_{60})$ with $k$ different $\rho_i$ (for $k=1,\ldots, 6$) is almost like finding symmetrically-different 6-bead bracelets where the beads are colored with $k$ colors.  The latter is a classic combinatorics problem and can be solved via Burnside's Lemma or Polya enumeration.  But these two problems are not exactly the same; using numbers for colors, Polya enumeration would distinguish $(111223)$ and $(222113)$ as differently-colored bracelets, but these would equate to the same kind of folding angle symmetry for $(G_{60},P_{60})$, one with three consecutive, equal folding angles, then two equal folding angles, then a third folding angle.  Therefore we need to enumerate and classify not just the different $k$-colored bracelets with six beads, but the \textit{$k$-colored bracelet patterns} on six beads.

Polya enumeration is not enough to solve this problem; while it can be a helpful tool, one needs to examine the individual terms in the cycle index generating function to see if folding symmetries are being counted multiple times.  This can be done by hand, but we double-checked by writing Python code to conduct a thorough enumeration.  Then for each type of bracelet pattern, we used Equation~\eqref{eq:2nd-order} to check if a rigid folding existed in a neighborhood of the origin; if solving the system of equations that \eqref{eq:2nd-order} generated led to a null or complex solution, then a rigid folding would not exist for that pattern.

\begin{table} 
\caption{Bracelet colorings and rigid folding symmetries for $G_{60}$}
\label{table1}
\begin{tabular}{cccc}
\hline
\# colors $k$ & $k$-color bracelet & rigid  & (symmetric pattern, folding name, DOF) \\
 & patterns & foldings \\
\hline
1 & 1 & 0 &  \\
2 & 7 & 2 & (121212, trifold, 1), (122122, bow tie, 1)\\
3 & 14 & 1 & (123123, opposites, 2) \\
4 & 10 & 2 & (123432, igloo, 2), (112234, two pair, 1)\\
5 & 3 & 1 & (112345, almost general, 2)\\
6 & 1 & 1 & (123456, fully general, 3)\\\hline
\end{tabular}
\vspace*{-4pt}
\end{table}

Our results are shown Table~\ref{table1} and illustrated in Figure~\ref{fig:symmcases} (aside from the trifold, which was previously shown in Figure~\ref{fig:trifold}).  It may seem surprising that so few rigid folding patterns exist from the numerous $k$-color bracelet patterns, but many of them reduce to the main folding patterns we discovered: trifold, bowtie, opposites, igloo, and two pair (to be described in Section~\ref{sec4}).  For example, the 3-color bracelet pattern $(123232)$, which converts to folding pattern $(\rho_1,\rho_2,\rho_3,\rho_2,\rho_3,\rho_2)$, generates equations in \eqref{eq:2nd-order} that simplify to $\rho_3=\rho_1$ and the trifold equations \eqref{eq:trifold}.

\section{Equations for symmetric modes}\label{sec4}

We now describe the equations for symmetrically-different rigid foldings of degree-6 vertices obtained from Equation~\eqref{eq:2nd-order}, generalizing them beyond the all-$60^\circ$-angle case when possible while retaining the folding angle symmetry.  Note that the local approximation of Equation~\eqref{eq:2nd-order} does not guarantee its extension, globally, to the whole configuration space, and therefore when such a local-to-global extension is possible we will aim to provide a proof that is independent of Equation~\eqref{eq:2nd-order}. 

\subsection{Generalized trifold, (121212)}\label{sec:gentrifold}

We can generalize the trifold rigid origami from Example~\ref{ex:trifold} by letting the sector angles of the crease pattern alternate $\beta$ and $2\pi/3-\beta$, as in Figure~\ref{fig:gentrifold}(a), while still having the folding angle pattern $(\rho_1,\rho_2,\rho_1,\rho_2,\rho_1,\rho_2)$, meaning that the rigid folds will still have $120^\circ$ rotational symmetry about the vertex.  The 2nd-order approximation of the folding equations near the origin in this case extend globally and give us the following:

\begin{figure}
    \centering
    \includegraphics[width=\linewidth]{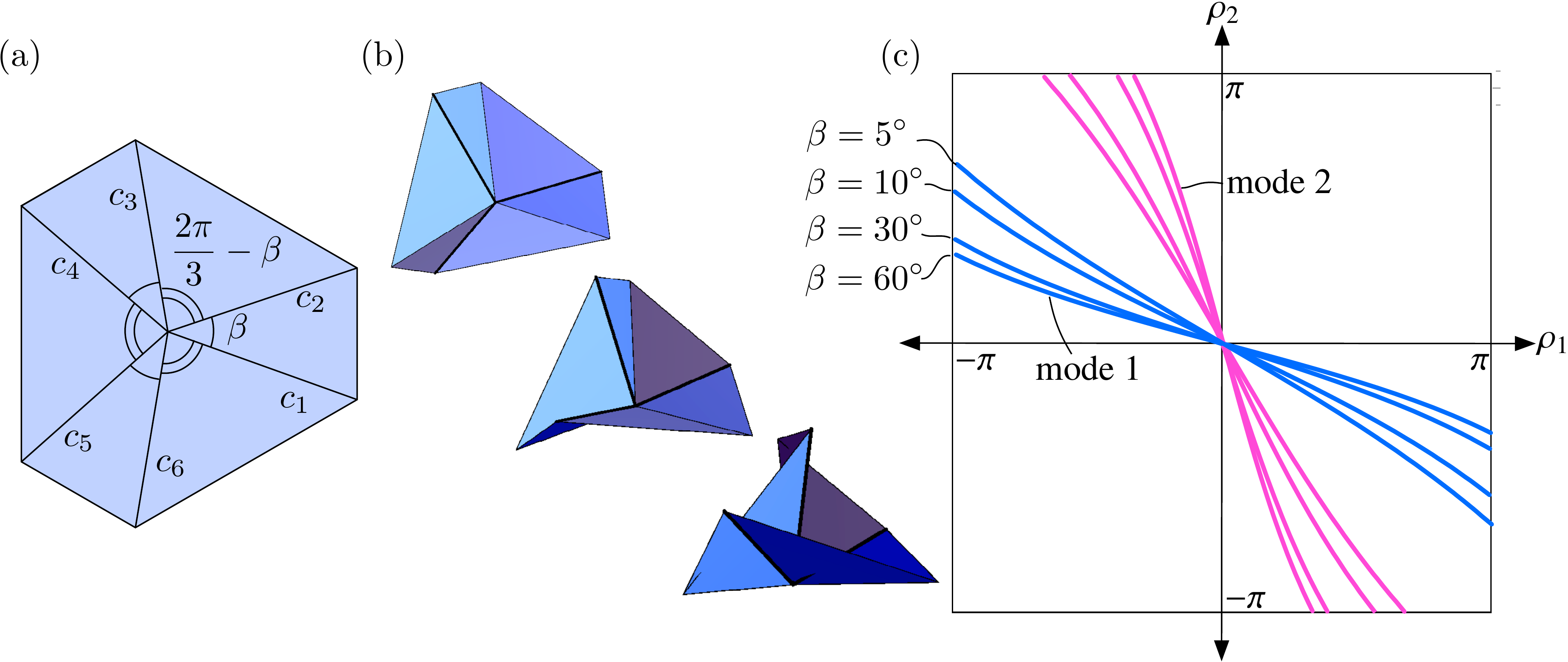}
    \caption{(a) The generalized trifold crease pattern.  (b) A sequence of mode 1 rigid foldings.  (c) Configuration space of the generalized trifold for a few $\beta$ values.}
    \label{fig:gentrifold}
\end{figure}

\begin{theorem}\label{thm:gentrifold}
The generalized trifold crease pattern in Figure~\ref{fig:gentrifold}(a), rigidly folding with folding angle symmetry $(\rho_1,\rho_2,\rho_1,\rho_2,\rho_1,\rho_2)$, has two rigid folding modes.  Mode 1 has
\begin{equation}\label{eq:gentrifold}
    \tan\frac{\rho_1}{4} = -\left(
    \cos\beta + \sqrt{3}\sin\beta + \sqrt{2\sin\beta (\sqrt{3}\cos\beta + \sin\beta)}\right)\tan\frac{\rho_2}{4},
\end{equation}
while the mode 2 equation is the same but with $\rho_1$ and $\rho_2$ reversed.
\end{theorem}

Note that if $\beta=60^\circ$ then Equation~\eqref{eq:gentrifold} becomes Equation~\eqref{eq:trifold} of the trifold in Example~\ref{ex:trifold}.

\begin{proof}
We follow a similar approach to the proof of Equation~\eqref{eq:trifold} given in \cite{Tachi2}.
By the $120^\circ$ rotational symmetry of the crease pattern and folded state, the matrix Equation~\eqref{eq:R=I} simplifies to 
$$
(R_x(\rho_1)R_z(\beta)R_x(\rho_2)R_z(2\pi/3-\beta))^3=I_3,
$$
where we assume that the paper lies in the $xy$-plane, $c_1$ is along the positive $x$-axis, and $R_x(\theta), R_z(\theta)$ are the matrices that rotate about the $x$- and $z$-axes by $\theta$, respectively.  This is equivalent to saying that $M=R_x(\rho_1)R_z(\beta)R_x(\rho_2)R_z(2\pi/3-\beta)$ is a $120^\circ$ rotation of $\mathbb{R}^3$ about some axis.  A common fact (derived from Rodrigues' rotation matrix formula) is that the trace of such a rotation matrix will equal $1+2\cos(120^\circ) = 0$.  On the other hand, 
\begin{equation}\label{eq:gentri1}
\begin{split}
    \mbox{Tr}(M) & =  \cos\beta \sin\left(\beta-\frac{\pi}{6}\right) - \cos\left(\beta-\frac{\pi}{6}\right)\sin\beta (\cos\rho_1 -\cos\rho_2) + \\
   & \quad(1+\sin\left(\beta-\frac{\pi}{6}\right)\cos\beta)\cos\rho_1 \cos\rho_2 - 
    (\cos\beta + \sin\left(\beta-\frac{\pi}{6}\right))\sin\rho_1 \sin\rho_2 \\
   & = \left(\frac{C\sin\beta}{\sqrt{(\cos\beta-1)(S-1)}}\cos\frac{\rho_1+\rho_2}{2}-\sqrt{(\cos\beta-1)(S-1)}\cos\frac{\rho_1-\rho_2}{2}+1\right)\cdot \\
   & \quad\left(\frac{C\sin\beta}{\sqrt{(\cos\beta-1)(S-1)}}\cos\frac{\rho_1+\rho_2}{2}-\sqrt{(\cos\beta-1)(S-1)}\cos\frac{\rho_1-\rho_2}{2}-1\right)
\end{split}
\end{equation}
where $S=\sin(\beta-\pi/6)$ and $C=\cos(\beta-\pi/6)$.  The factorization of Tr$(M)$ in Equation~\eqref{eq:gentri1} is obtained from a sequence of trigonometric identity manipulations.  Since Tr$(M)=0$, we can see that the second factor of \eqref{eq:gentri1} does not represent our desired configuration space for the generalized trifold.  For example, if we set $\rho_1=\rho_2=0$, then this factor simplifies to the constant $-2$, which is never zero.  Therefore the first factor must equal zero, and the resulting equation has solution curves for $(\rho_1,\rho_2)$ as shown in Figure~\ref{fig:gentrifold}(c) for various values of $\beta$.  As seen in \cite{Tachi2}, this type of equation can be re-written in terms of $\tan(\rho_1/4)$ and $\tan(\rho_2/4)$.  Doing this results in the equivalent expression \eqref{eq:trifold} and its mode 2 counterpart.
\end{proof}

\subsection{Generalized bow tie, (112112)}\label{sec:bowtie}

\begin{figure}
    \centering
    \includegraphics[width=\linewidth]{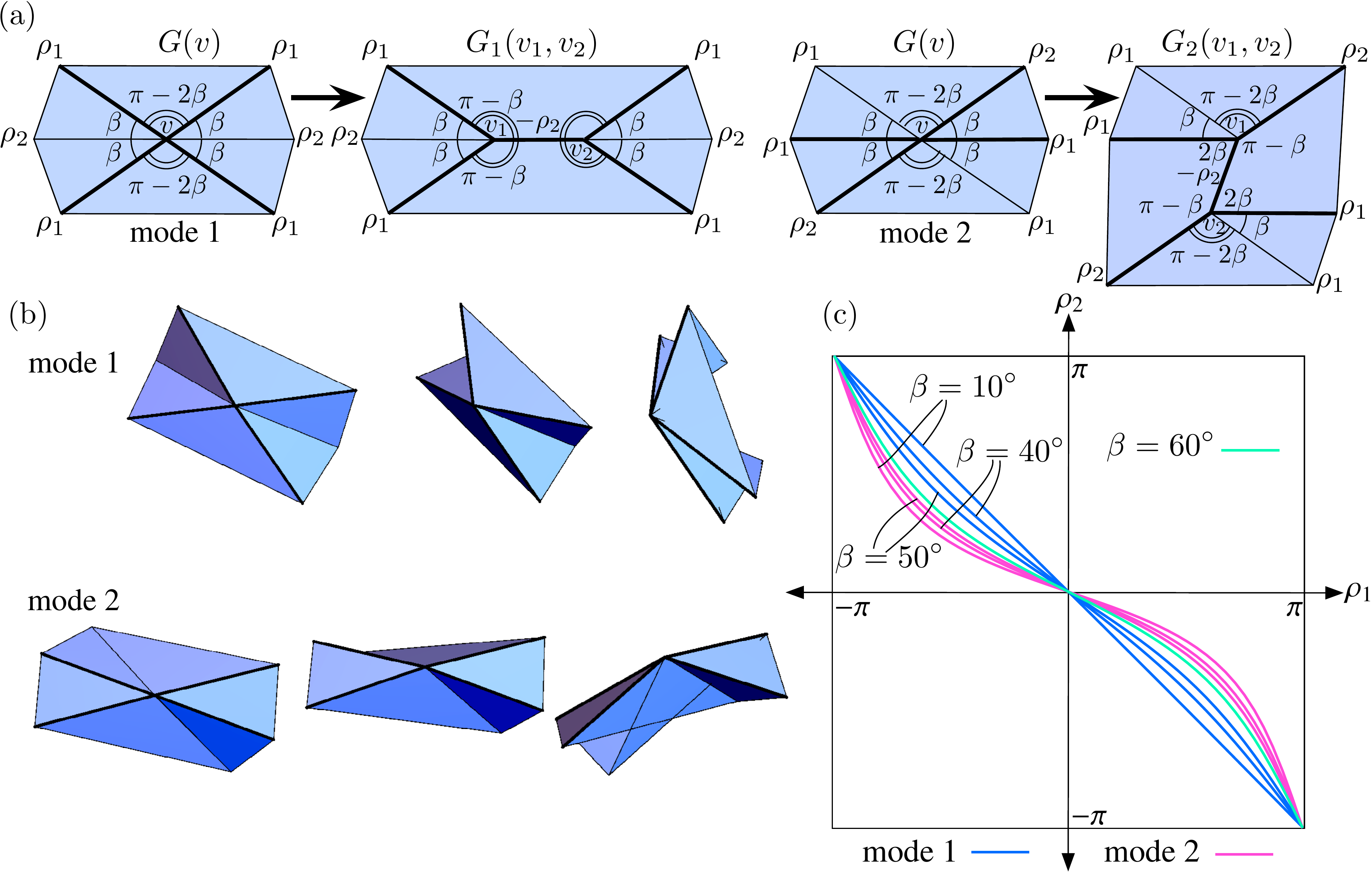}
    \caption{(a) The generalized bow tie crease patterns, modes 1 and 2.  (b) Rigid foldings of the bow ties.  (c) Configuration space of the generalized bow tie for a few $\beta$ values, where the modes converge when $\beta=60^\circ$.}
    \label{fig:genbowtie}
\end{figure}

The bow tie symmetry has folding angles in the pattern $(\rho_1,\rho_1, \rho_2, \rho_1, \rho_1,\rho_2)$.  When folded on $(G_{60},P_{60})$ there is only one folding mode up to rotation about the vertex.  If we generalize so that the crease pattern still has horizontal and vertical symmetry, as shown in Figure~\ref{fig:genbowtie}(a), we have two folding modes that still produce a bow tie shape (Figure~\ref{fig:genbowtie}(b)).  The mountain-valley (MV) assignment labeled mode 1 of the generalized bow tie appears in the well-studied waterbomb tessellation; see vertices $w_1$ and $w_3$ in the crease pattern of Figure~\ref{fig:waterbomb}(b).  Thus the folding angle equations of this case are already known and were first reported in \cite{Chen}.  The mode 2 MV assignment does not appear in the waterbomb tessellation and seems to be new.  For completeness we include the explicit folding angle equations for both modes in the following theorem.

\begin{theorem}\label{thm:genbowtie}
The generalized bow tie crease pattern in Figure~\ref{fig:genbowtie}(a), rigidly folding with folding angle symmetry $(\rho_1,\rho_1,\rho_2,\rho_1,\rho_1,\rho_2)$, has two rigid folding modes, parameterized by  
\begin{equation}\label{eq:genbowtie}
    \tan\frac{\rho_2}{2} = -\cos\beta \tan\frac{\rho_1}{2}\mbox{ for mode 1 and }
    \tan\frac{\rho_2}{2} = -\frac{1}{1+2\cos\beta}\tan\frac{\rho_1}{2}\mbox{ for mode 2.}
\end{equation}
\end{theorem}

Note that when $\beta=60^\circ$ both of the equations in \eqref{eq:genbowtie} become the same.

\begin{proof}
The result may be proved using the vertex-splitting technique from \cite{Tolman}, and indeed this is exactly what Zhang and Chen do in \cite{Zhang-Chen} to derive the mode 1 equation.  That is, in the mode 1 case we split the degree-6 bow tie vertex $v$ in the crease pattern $G(v)$ into two flat-foldable degree-4 vertices $v_1$ and $v_2$ in crease pattern $G_1(v_1,v_2)$ as in Figure~\ref{fig:genbowtie}(a). This new 2-vertex crease pattern must be kinematically equivalent to the mode 1 generalized bow tie, since if we shorten the segment $\overline{v_1v_2}$ in $G_1(v_1,v_2)$ the folding angles $\rho_1, \rho_2$ will not change.  Thus in the limit the $\rho_1,\rho_2$ folding angles in $G_1(v_1,v_2)$ must be the same as those in $G(v)$.  Furthermore, the two vertices in $G_1(v_1,v_2)$ are bird's feet and force the creases to fold with the desired bow tie symmetry.
Specifically, Theorem~\ref{thm:deg4} gives us that $\tan(\rho_2/2)=-\cos\beta \tan(\rho_1/2)$.  

In the mode 2 case, we split the vertex $v$ in $G(v)$ along a different line to produce two degree-4 flat-foldable vertices $v_1$ and $v_2$ in $G_2(v_1,v_2)$ with sector angles $(\beta,2\beta,\pi-\beta,\pi-2\beta)$, as shown in Figure~\ref{fig:genbowtie}(a). In $G_2(v_1,v_2)$ the creases bounding the sector angle $\beta$ must have different MV parity, and only one choice for these creases matches the desired bow tie symmetry (the one shown in Figure~\ref{fig:genbowtie}(a)).  This will be kinematically equivalent to the mode 2 rigid folding of $G(v)$, as in the mode 1 case.  In this MV assignment the degree-4 vertices in $G_2(v_1,v_2)$ are folding in mode 2 of Theorem~\ref{thm:deg4}, and therefore we have that  $\tan(\rho_2/2)=-(\sin(\beta/2)/\sin(3\beta/2))\tan(\rho_1/2) = -1/(1+2\cos\beta)\tan(\rho_1/2)$, as desired.
\end{proof}

\subsection{Generalized opposites, (123123)}\label{sec:opp}

The pattern $(123123)$ is the only 3-color bracelet pattern whose corresponding rigid folding symmetry for $G_{60}$ does not reduce to either the trifold, the bow tie, or simply does not rigidly fold.  We generalize this pattern in a way that matches the symmetry $(\rho_1,\rho_2,\rho_3,\rho_1,\rho_2,\rho_3)$, as shown in Figure~\ref{fig:genopp}(a), denoting the crease pattern by $O_{\alpha,\beta}$.  Note that if $\beta=\pi-2\alpha$ then this crease pattern becomes that of the generalized bow tie, and if $\rho_3=\rho_2$ then we get mode 1 of the generalized bow tie and if $\rho_3=\rho_1$ or $\rho_2=\rho_1$ then we get mode 2.   Therefore we would expect that the configuration space of the generalized opposites pattern should contain the configuration spaces of the different modes of the generalized bow tie as subsets.  This immediately suggests that if we consider the generalized opposites configuration space to be points $(\rho_1,\rho_2,\rho_3)\in \mathbb{R}^3$, then ${\cal C}(O_{\alpha,\beta})$ cannot be a 1-dimensional curve, but must be a 2-dimensional surface.  In other words, the kinematics of the generalized opposites must have two DOF.

We verify this by considering the 2nd-order approximation of the rigid folding around the origin, as in Section~\ref{sec2}.  We set
$c_i=\langle \cos\theta_i,\sin \theta_i,0\rangle$ where $\theta_1=0$, $\theta_2=\alpha$, $\theta_3=\alpha+\beta$, $\theta_4=\pi$, $\theta_5=\pi+\alpha$, and $\theta_6=\pi+\alpha+\beta$.
Then the first matrix sum in \eqref{eq:2nd-order} becomes 
$$
\begin{pmatrix}
0 & -A & 0\\
A & 0 & 0\\
0 & 0 & 0
\end{pmatrix}\mbox{ where } A=2(\sin\alpha \rho_1'(0)\rho_2'(0) + \sin\beta \rho_2'(0)\rho_3'(0)+\sin(\alpha+\beta) \rho_1'(0)\rho_3'(0)).$$

Setting this equal to zero and reparameterizing with the Weierstrass substitution gives the  folding angle approximation near the origin
\begin{equation}\label{eq:opp1}
\sin\alpha \tan\frac{\rho_1}{2}\tan\frac{\rho_2}{2}
+\sin\beta \tan\frac{\rho_2}{2}\tan\frac{\rho_3}{2} + \sin(\alpha+\beta)\tan\frac{\rho_1}{2}\tan\frac{\rho_3}{2}=0.
\end{equation}
Surprisingly, Equation~\eqref{eq:opp1} extends to the whole configuration space $-\pi\leq \rho_1,\rho_2,\rho_3\leq\pi$.  That is, we can express any of the $\rho_i$ in terms of the other two and generate rigid folding simulations like those in Figure~\ref{fig:genopp}(b).  Note that a proof of this that does not rely on the 2nd-order approximation is not possible with the vertex-splitting technique of the Section~\ref{sec4}\ref{sec:bowtie} (or the parallel pleat transform of the Section~\ref{sec4}\ref{sec:igloo} below) because reducing $O_{\alpha,\beta}$ to a crease pattern with only degree-4 vertices will necessarily result in a 1-DOF system, and we know that $O_{\alpha,\beta}$ is 2-DOF.  

In Figure~\ref{fig:genopp}(c) we see the configuration space ${\cal C}(O_{\alpha,\beta})$ given by  Equation~\ref{eq:opp1}.  On this surface we have drawn the three different bow tie special cases of $O_{\alpha,\beta}$, as mentioned previously.  Indeed, if we set $\beta=\pi-2\alpha$ in \eqref{eq:opp1} and either $\rho_3=\rho_2$, $\rho_3=\rho_1$, or $\rho_2=\rho_1$ we obtain the folding angle equations for the generalized bow tie in Theorem~\ref{thm:genbowtie}.

\begin{figure}
    \centering
    \includegraphics[width=\linewidth]{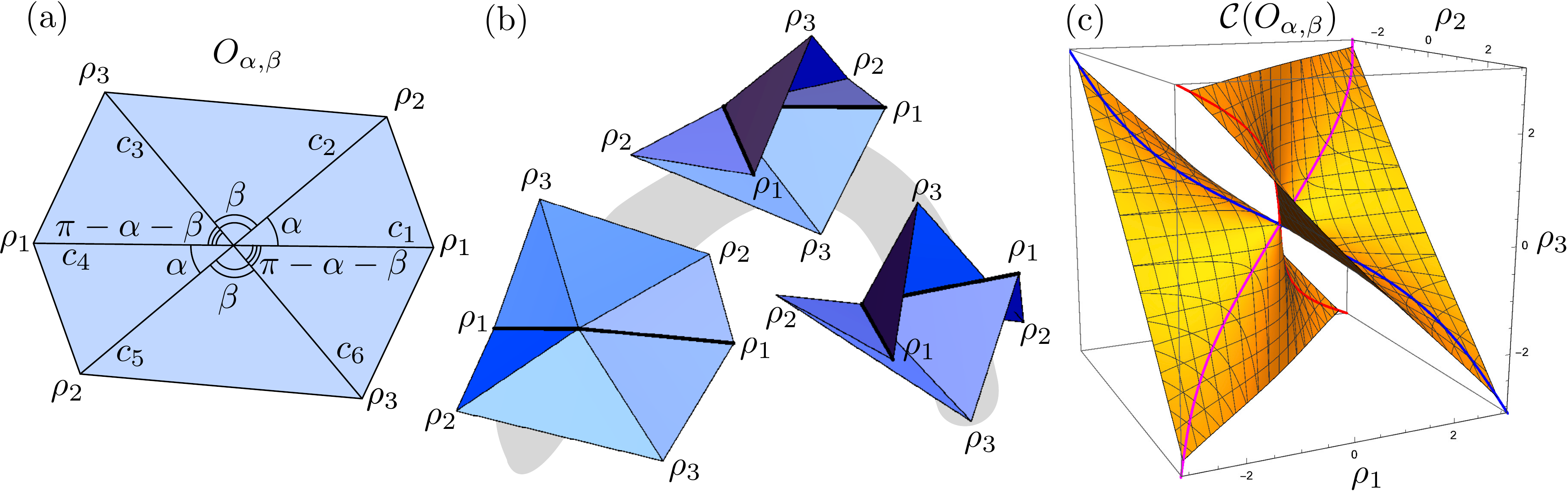}
    \caption{(a) The generalized opposites crease pattern.  (b) Rigid foldings of the opposites pattern.  (c) Configuration space of the opposites pattern with curves (in purple, pink, and red) showing the three ways that the bow tie symmetry is a special case of the opposites symmetry.}
    \label{fig:genopp}
\end{figure}

\subsection{Generalized igloo, (123432)}\label{sec:igloo}

\begin{figure}
    \centering
    \includegraphics[width=\linewidth]{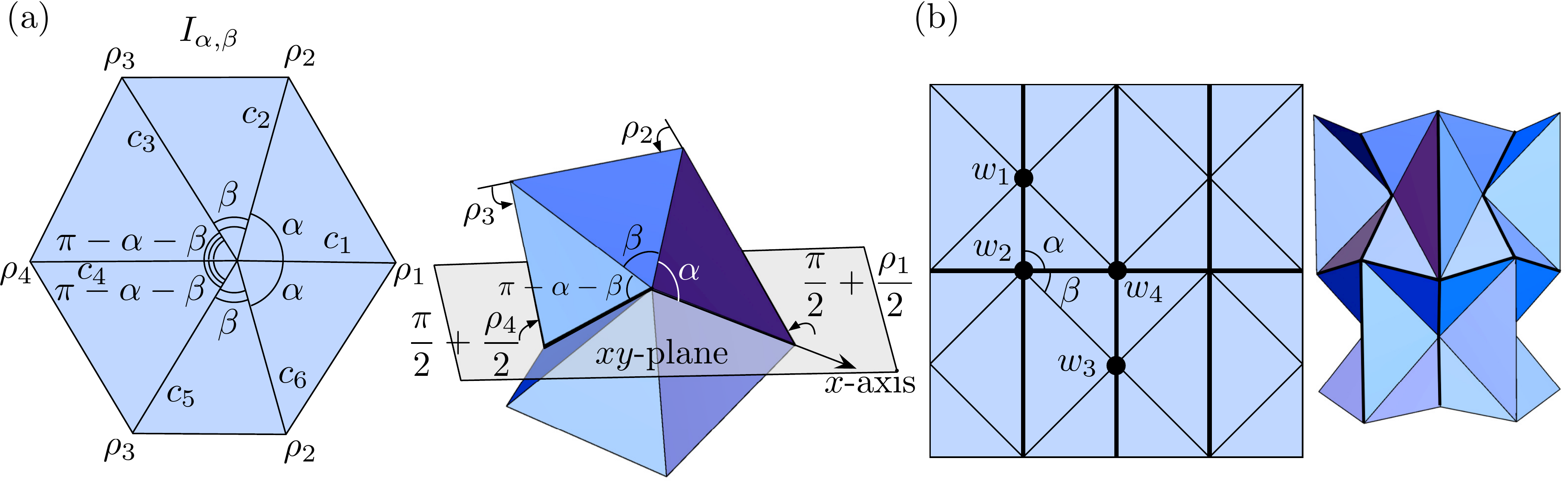}
    \caption{(a) The plane symmetry of a folded $I_{\alpha,\beta}$. (b) The waterbomb tessellation crease pattern and rigid folding simulation.  All the vertices are degree-6 and have either bow tie  ($w_1$ and $w_3$) or igloo symmetry ($w_2$ and $w_4$).}
    \label{fig:waterbomb}
\end{figure}

The igloo from Figure~\ref{fig:symmcases} and Table~\ref{table1} has folding angles of the form $(\rho_1,\rho_2,\rho_3,\rho_4,\rho_3,\rho_2)$, or $(123432)$ bracelet pattern.  To generalize this, we maintain the line of reflection symmetry of this pattern and allow $\alpha$ to be the sector angles between creases $c_1,c_2$ and $c_1,c_6$ and $\beta$ to be the angle between $c_2,c_3$ and $c_5,c_6$.  We call this crease pattern $I_{\alpha,\beta}$; see Figure~\ref{fig:waterbomb}(a). 

Like the bow tie, the igloo symmetry pattern appears in the waterbomb tessellation pattern, as seen at vertices $w_2$ and $w_4$ in Figure~\ref{fig:waterbomb}(b).  Explicit folding angle equations for this waterbomb tessellation instance of the igloo pattern, even when generalized to $I_{\alpha,\beta}$, can be found in \cite{Chen}.  However, this does not tell the whole story of $I_{\alpha,\beta}$.  As stated in \cite{Chen}, when the waterbomb tessellation is folded symmetrically, so that the bow tie and igloo vertices rigidly fold symmetrically, then its rigid folding is only 1-DOF.  But the generalized igloo $I_{\alpha,\beta}$ on its own is a 2-DOF rigid origami.
We capture this, as well as equations for the whole configuration space of $I_{\alpha,\beta}$ in the following theorem.

\begin{theorem}\label{thm:iglooconfig}
The symmetric degree-6 crease pattern $I_{\alpha,\beta}$ is a 2-DOF rigid origami, and the configuration space ${\cal C}(I_{\alpha,\beta})$ is given by
\begin{equation}\label{eq:iglooconfig1}
\tan\frac{\rho_1}{2} =
\frac{\splitfrac
{\sin\alpha\sin\beta\sin(\alpha+\beta)\cos\rho_3
+\cos\beta(\sin\alpha\cos(\alpha+\beta)-\cos\alpha\sin(\alpha+\beta)\cos\rho_2\cos\rho_3)}
{+\cos\alpha(\sin\beta\cos(\alpha+\beta)\cos\rho_2+\sin(\alpha+\beta)\sin\rho_2\sin\rho_3)}
}
{\sin\beta\cos(\alpha+\beta)\sin\rho_2 - \sin(\alpha+\beta)(\cos\beta\sin\rho_2\cos\rho_3+\cos\rho_2\sin\rho_3)}
\end{equation}
and
\begin{equation}\label{eq:iglooconfig2}
\tan\frac{\rho_4}{2} =
\frac{\splitfrac
{\cos^2\alpha \sin(2(\alpha+\beta))\sin^2\frac{\rho_3}{2} 
+\sin\alpha \cos(\alpha+\beta)(2\sin\alpha \sin(\alpha+\beta) \cos\rho_2 \sin^2\frac{\rho_3}{2} }
{- \sin\alpha\cos\alpha(\cos\rho_2-1)(\cos^2(\alpha+\beta)\cos\rho_3 + \sin^2(\alpha+\beta))
+ \sin\rho_2 \sin\rho_3)}
}
{\sin\alpha \sin\rho_2 \cos\rho_3 (\sin\alpha \cos\beta \cos\rho_2 + \cos\alpha \sin\beta)\sin\rho_3}.
\end{equation}
\end{theorem}

\begin{proof}
To see that the generalized igloo is a 2-DOF rigid origami, imagine that we place the vertex of $I_{\alpha,\beta}$ at the center of a unit radius sphere $S$  so that the folded origami intersects $S$ at a closed spherical linkage. Fix the sector of $I_{\alpha,\beta}$ between creases $c_1, c_2$ and choose any $\rho_1,\rho_2\in [-\pi,\pi]$.  This will determine the position of the creases $c_3$ and $c_5$, and if we draw two circles of radii $\pi-\alpha-\beta$ at these creases' endpoints on $S$, then the intersection of these circles will determine either 0, 1, or 2 possible locations for the crease $c_4$.  Thus any pair $(\rho_1,\rho_2)$ will determine a rigid folding of $I_{\alpha,\beta}$, and any rigid folding of $I_{\alpha,\beta}$ will give us such a pair $(\rho_1,\rho_2)$, making this a 2-DOF system.

As noted in \cite{Chen}, all rigid foldings of $I_{\alpha,\beta}$ are reflection-symmetric about the plane that contains the creases $c_1$ and $c_4$.  We let this plane be the $xy$-plane and position the crease $c_1$ to be on the positive $x$-axis, as shown in Figure~\ref{fig:waterbomb}(a).  Then, keeping $c_1$ fixed, we imagine folding the sectors of paper between the creases $c_1$, $c_2$, $c_3$, and $c_4$.  To replicate the folding angle $\rho_1$ at $c_1$, the $c_1c_2$ sector would have to lift from the $xy$-plane by an angle of $\pi/2+\rho_1/2$, while the folding angles at $c_2$ and $c_3$ will simply be $\rho_2$ and $\rho_3$, respectively.  We let $c_4=\langle -1,0,0\rangle$. Then the image of $c_4$ after this folding will be
\begin{equation}\label{eq:iglooprod1}
R_x(\frac{\pi}{2}+\frac{\rho_1}{2})R_z(\alpha)R_x(\rho_2)R_z(-\alpha)R_z(\alpha+\beta)R_x(\rho_3)R_z(-(\alpha+\beta))c_4.
\end{equation}
In order for the folding of $I_{\alpha,\beta}$ to be symmetric about the $xy$-plane, we must have that the $z$-coordinate of~\eqref{eq:iglooprod1} is zero.  This simplifies to
\begin{equation}\label{eq:iglooprod2}
\begin{multlined}
\cos\frac{\rho_1}{2}(
\cos\alpha(\cos\beta\cos\rho_2\cos\rho_3 - \sin\rho_2\sin\rho_3) 
-\sin\alpha\sin\beta\sin(\alpha+\beta)\cos\rho_3 \\
- \cos(\alpha+\beta)((\sin\alpha\cos\beta+\cos\alpha\sin\beta\cos\rho_2)) ) \\
= \sin\frac{\rho_1}{2}(
 \sin(\alpha+\beta) (\cos\beta\sin\rho_2\cos\rho_3+\cos\rho_2\sin\rho_3)
-\sin\beta\cos(\alpha+\beta)\sin\rho_2).
\end{multlined}
\end{equation}
Isolating the $\rho_1$ terms in~\eqref{eq:iglooprod2} to create $\tan(\rho_1/2)$ and simplifying a bit further gives~\eqref{eq:iglooconfig1}.  Performing the same computation on $I_{\alpha,\beta}$ except flipped so that the crease $c_4$ is fixed to the positive $x$-axis gives an equation on $(\rho_2,\rho_3,\rho_4)$, which is~\eqref{eq:iglooconfig2}.
\end{proof}

\begin{figure}
    \centering
    \includegraphics[width=5in]{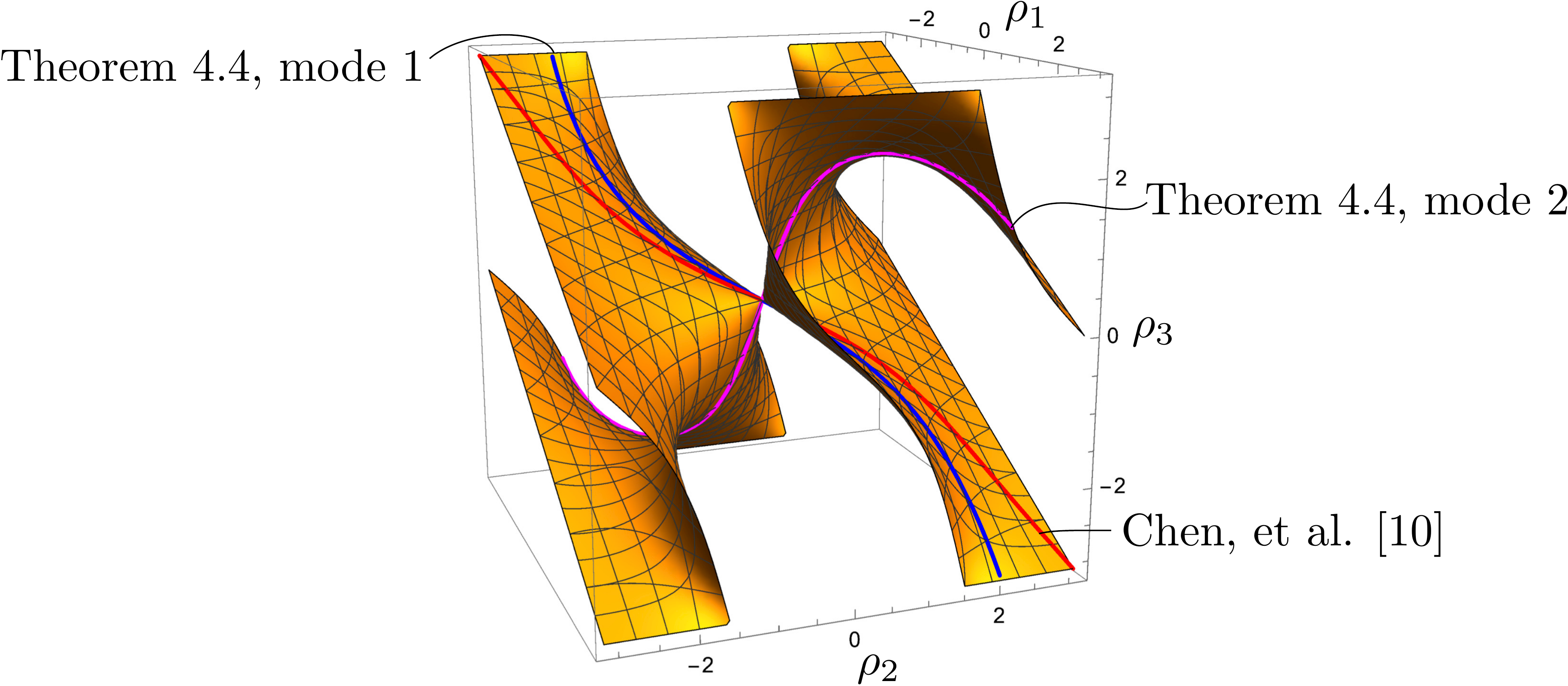}
    \caption{The configuration space ${\cal C}(I_{60^\circ, 60^\circ})$ in $(\rho_1,\rho_2,\rho_3)$ parameter space. The blue and pink curves are the mode 1 and mode 2 curves from Theorem~\ref{thm:genigloo}, respectively, while the red curve is the rigid folding of $I_{\alpha,\beta}$ in a symmetric folding of the waterbomb tessellation derived by Chen, et al., 2016  (equations (2.4a) and (2.4b) in \cite{Chen}).}\label{fig:iglooconfig}
\end{figure}

The equations \eqref{eq:iglooconfig1} and \eqref{eq:iglooconfig2} can be used to create 2-DOF rigid folding simulations of $I_{\alpha,\beta}$.  The configuration space ${\cal C}(I_{\alpha,\beta})$ for $\alpha=\beta=60^\circ$ that we obtain from these equations in the $(\rho_1,\rho_2,\rho_3)$ parameter space is shown in Figure~\ref{fig:iglooconfig} (we actually used the singularity-free Equation~\eqref{eq:iglooprod2} to compute this image).

\begin{figure}
    \centering
    \includegraphics[width=\linewidth]{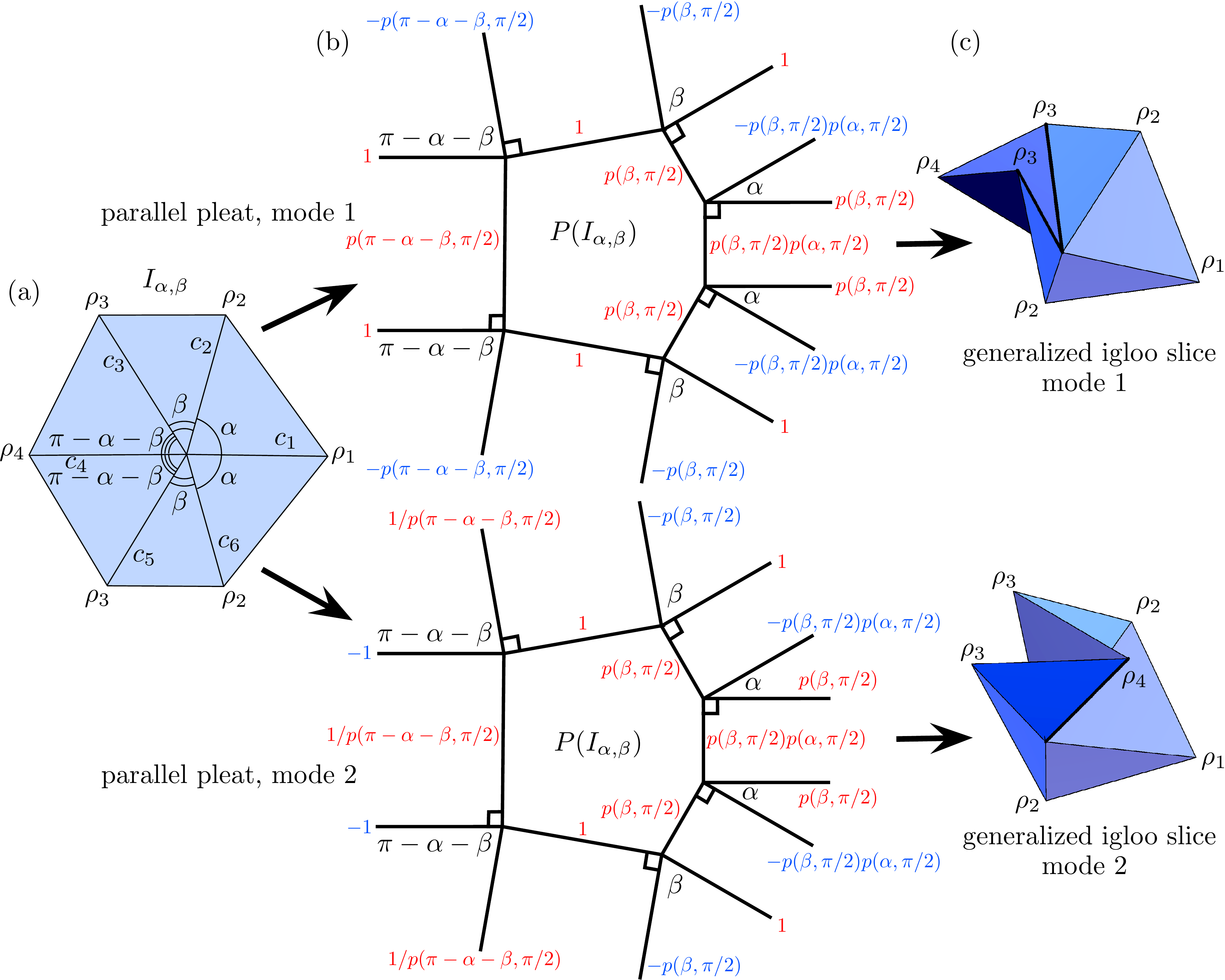}
    \caption{(a) The generalized igloo crease pattern.  (b) The parallel pleat transform applied to the igloo crease pattern, with creases labeled with their folding angle multipliers for modes 1 and 2 (red multipliers are mountain and blue are valley creases).  (c) Illustrations of rigid folding modes 1 and 2 of the generalized igloo.}
    \label{fig:genigloo}
\end{figure}

We see in Theorem~\ref{thm:iglooconfig} a glimmer of the Weierstrass substitution in the 2-DOF system $I_{\alpha,\beta}$, although it is not as apparent or elegant as in the generalized opposites pattern $O_{\alpha,\beta}$.  However, several 1-DOF slices of the configuration space ${\cal C}(I_{\alpha,\beta})$ can be found that show instances where the Weierstrass substitution, like the degree-4 and trifold cases, reveals linear relationships between some folding angles.

One way to find such 1-DOF slices is to use the \textit{parallel pleat transform} from \cite{HullTachi1} (also known as the \textit{double line method}).  This method takes as input a rigidly-foldable crease pattern $C$ and creates a new crease pattern $P(C)$ where each crease $c$ of $C$ is replaced in $P(C)$ by a pair of parallel creases $p_1(c)$, $p_1(c)$ (a pleat) and the vertices of degree $n$ in $C$ are replaced by an $n$-sided polygon in $P(C)$.  This is done in such a way that the vertices in $P(C)$ are all degree-4 and flat-foldable, so that the folding angle equations of Theorem~\ref{thm:deg4} can be used to rigidly fold $P(C)$.  Thus, the folding angle of $c$ in $C$ will equal the sum of the folding angles of $p_1(c)$ and $p_2(c)$ in $P(C)$.  

This parallel pleat transform has additional constraints; see \cite{HullTachi1} for more details.  But in the case of the generalized igloo it works very well and is shown in Figure~\ref{fig:genigloo}(b), where the folding angle multipliers of the creases in $P(I_{\alpha,\beta})$ are shown and lead to two different modes for this crease pattern (one where the "igloo door" pops in and one where it pops out). These multipliers along with Theorem~\ref{thm:deg4} give us folding angle relationships for $\rho_1,\ldots,\rho_4$ as shown in the following Theorem.

\begin{theorem}\label{thm:genigloo}
The generalized igloo crease pattern $I_{\alpha,\beta}$ in Figure~\ref{fig:genigloo}, rigidly folding with symmetry $(\rho_1,\rho_2,\rho_3,\rho_4,\rho_3,\rho_2)$, has two rigid folding modes parameterized by
$$ 
    \tan\frac{\rho_1}{4}  = \frac{1-\tan\frac{\beta}{2}}{1+\tan\frac{\beta}{2}} \tan\frac{\rho_4}{4}, \ \
    \tan\left(\frac{\rho_2}{2}-\frac{\rho_4}{4}\right)  = - \frac{1-\tan\frac{\alpha}{2}}{1+\tan\frac{\alpha}{2}}
    \frac{1-\tan\frac{\beta}{2}}{1+\tan\frac{\beta}{2}} \tan\frac{\rho_4}{4}$$
$$    \tan\frac{\rho_3}{2}  = -\frac{\sin\frac{\alpha}{2} \sin\frac{\rho_4}{2}}
    {\cos\frac{\alpha}{2}+\cos\frac{\rho_4}{2}\sin\left(\frac{\alpha}{2}+\beta\right)}
$$
for mode 1,  and for mode 2
$$ 
    \tan\frac{\rho_1}{4}  = \frac{1-\tan\frac{\beta}{2}}{1+\tan\frac{\beta}{2}} \tan\frac{\rho_4}{4}, \ \
    \tan\left(\frac{\rho_2}{2}-\frac{\rho_4}{4}\right)  =  -\frac{1-\tan\frac{\alpha}{2}}{1+\tan\frac{\alpha}{2}}
    \frac{1-\tan\frac{\beta}{2}}{1+\tan\frac{\beta}{2}} \tan\frac{\rho_4}{4}$$
$$    \tan\frac{\rho_3}{2}  = \frac{2\cos\frac{\alpha}{2}\sin\frac{\rho_4}{2}}
    {\cos\left(\frac{\alpha}{2}+\beta+\frac{\rho_4}{2}\right)+\cos\left(\frac{\alpha}{2}+\beta-\frac{\rho_4}{2}\right)-2\sin\frac{\alpha}{2}}.
$$
\end{theorem}

\begin{proof}
In the parallel pleat transform, we chose our pleats in $P(I_{\alpha,\beta})$ to be perpendicular to the central polygon.  This means that the folding angle multipliers of the degree-4 flat-foldable vertices in $P(I_{\alpha,\beta})$ will all be of the form $p(x,\pi/2)$ for $x\in\{\alpha,\beta,\pi-\alpha-\beta\}$.  See Figure~\ref{fig:genigloo}(b) for reference.

In mode 1, we choose the parallel pleats made from $c_4$ to have folding multiplier 1, so if $s$ is the folding angle of these two creases in $P(I_{\alpha,\beta})$, we have that $\rho_4=2s$.  Next, the parallel pleats made from crease $c_1$ in $I_{\alpha,\beta}$ have, according to Theorem~\ref{thm:deg4}, folding angles $2\arctan(p(\beta,\pi/2)\tan(s/2))$.  Therefore, 
$$\rho_1= 4\arctan(p(\beta,\pi/2)\tan\frac{\rho_4}{4}\ \Rightarrow \tan\frac{\rho_1}{4} = p(\beta,\pi/2)\tan\frac{\rho_4}{4},$$
as desired.  For crease $c_2$ we obtain $\rho_2 = s -2\arctan(p(\beta,\pi/2)p(\alpha,\pi/2)\tan(s/2))$, and for crease $c_3$ we have $\rho_3 = -2\arctan(p(\beta,\pi/2)\tan(s/2)) - 2\arctan(p(\pi-\alpha-\beta,\pi/2)\tan(s/2))$. When simplified, these produce the mode 1 equations as stated in the Theorem.  The equations for mode 2 follow similarly.
\end{proof}

\begin{figure}
    \centering
    \includegraphics[width=\linewidth]{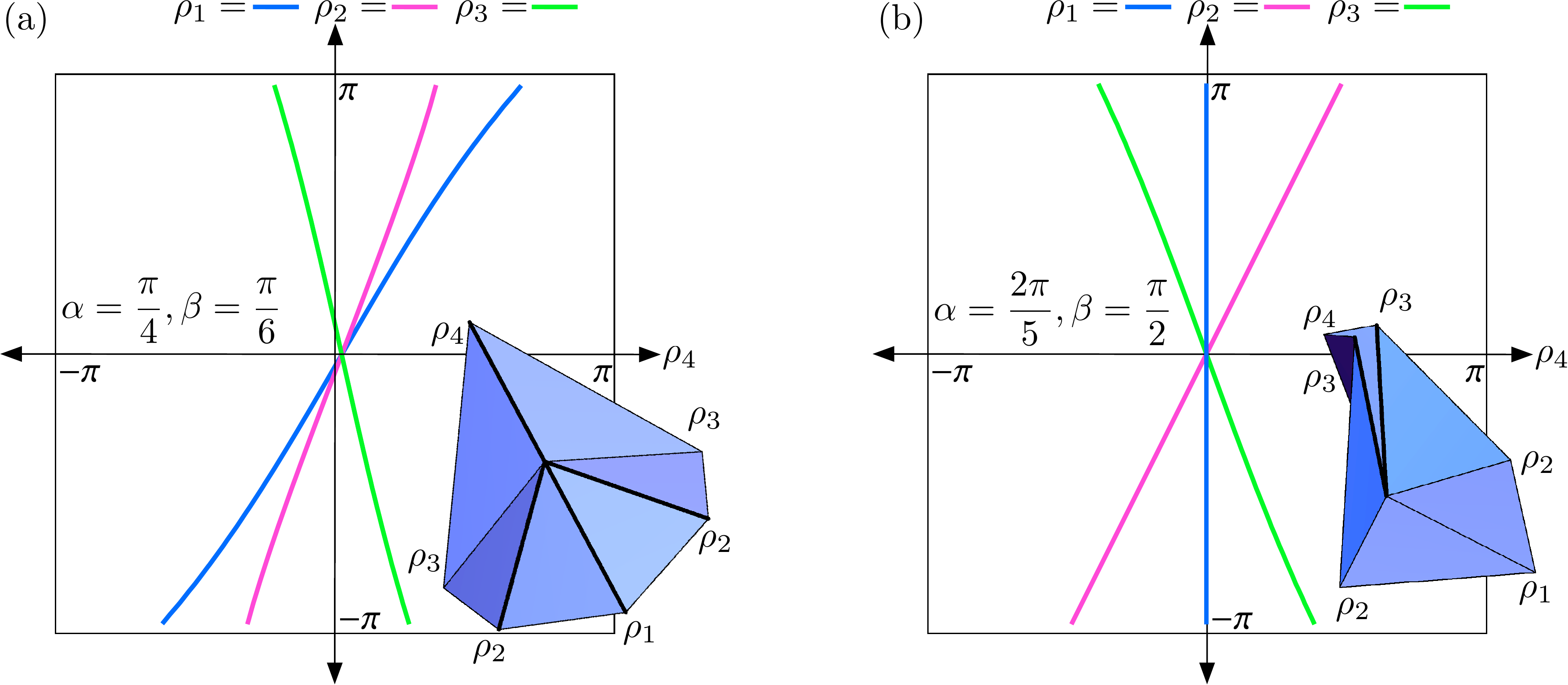}
    \caption{Generalized igloo example configuration spaces in mode 1 with (a)  $\alpha=\pi/4$, $\beta=\pi/6$ and  (b) $\alpha=2\pi/5$, $\beta=\pi/2$. }
    \label{fig:iglooconfig}
\end{figure}

A few notes:
\begin{enumerate}
    \item As noted in \cite{HullTachi1}, the parallel pleat transform offers an explanation for why sometimes the modified Weierstrass substitution $\tan(\rho_i/4)$ appears in rigid folding angle equations, as it arises naturally when adding the two folding angles of a pleat that are parameterized with $2\arctan(k \tan(t_i/2))$ for some constant $k$ and where $\rho_i=2t_i$.

    \item The folding angles $\rho_1$, $\rho_2$, and $\rho_3$ are all expressed in terms of $\rho_4$ in Theorem~\ref{thm:genigloo}, and thus the configuration space of this case, which is four-dimensional, can be visualized by $(\rho_4,\rho_i)$ slices, as shown in Figure~\ref{fig:iglooconfig}.
    
    \item While parameterizations with Weierstrass substitutions $\rho_i/2$ and $\rho_i/4$ are successful  in Theorem~\ref{thm:genigloo} for $(\rho_4,\rho_1)$ and $(\rho_4,\rho_2)$, it does not seem to be a successful way to completely express the $(\rho_4,\rho_3)$ relationship (although the equations given look tantalizingly close).
        
    \item When $\beta=\pi/2$ the first equation in modes 1 and 2 reduces to $\rho_1=0$, meaning that crease $c_1$ is never folded and so $I_{\alpha,\pi/2}$ is actually a rigid folding vertex of degree 5, as shown in Figure~\ref{fig:iglooconfig}(b).  Also in this case the second equation in Theorem~\ref{thm:genigloo} becomes $\rho_2=\rho_4/2$ in mode 1 and $\rho_2=-\rho_4/2$ in mode 2, giving a direct linear relationship between these folding angles.  (This also happens in the case where $\alpha=\pi/2$.)
\end{enumerate}

The curves that modes 1 and 2 in Theorem~\ref{thm:genigloo} make in the configuration space ${\cal C}(I_{\alpha,\beta})$ are shown (in blue and pink, respectively) in Figure~\ref{fig:iglooconfig}.  Also in this Figure we show in red a different 1-DOF curve discovered by Chen et al. in \cite{Chen} that captures the kinematics of $I_{\alpha,\beta}$ in a symmetric folding of the waterbomb tessellation (specifically, equations (2.4a) and (2.4b) in \cite{Chen}, which also demonstrate the prevalence of the Weierstrass substitution).

\begin{figure}
    \centering
    \includegraphics[width=\linewidth]{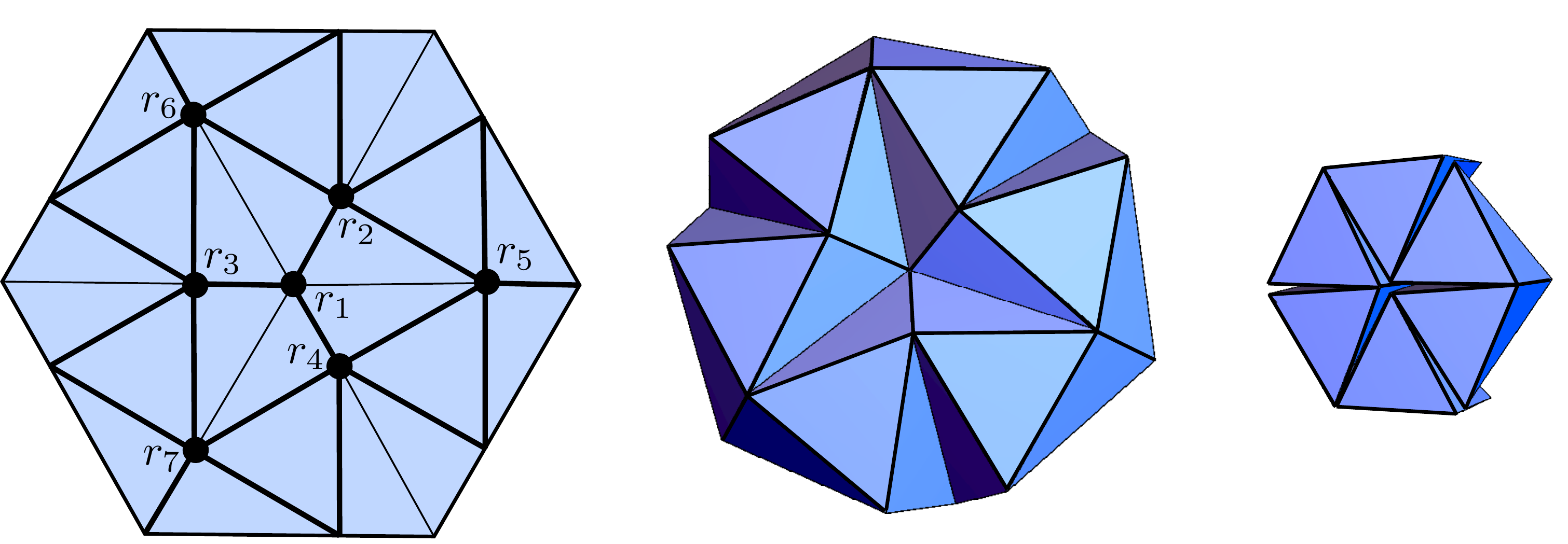}
    \caption{A rigid folding simulation of the Resch triangle twist tessellation using a combination of trifold and igloo equations. }
    \label{fig:Resch}
\end{figure}

\begin{example}[Resch triangle twist tessellation]\label{ex:Resch}
As an example of how the equations we've seen thus far can be applied to a larger crease pattern, we consider the triangle twist tessellation shown in Figure~\ref{fig:Resch}. This was designed by Ron Resch in the 1960s \cite{Resch,Resch2} and has been studied extensively for rigid folding and metamaterial purposes \cite{Goran,Tachi3,FufuYang}.  

As a rigid folding, the Resch pattern has many degrees of  freedom \cite{FufuYang}.  But if we insist on folding it with $120^\circ$ rotational symmetry about the center vertex $r_1$, then it becomes 1-DOF.  The vertex $r_1$ is a trifold rigid folding with $\beta=60^\circ$, and thus its rigid folding follows  (say, in mode 1 of) Equation~\eqref{eq:trifold}.  The folding angles for vertices $r_2$-$r_4$ can then be given by the 1-DOF generalized igloo mode 2 equations in Theorem~\ref{thm:genigloo}.  This provides the folding angle inputs needed to determine the folding angles for vertices $r_5$-$r_7$ using the 2-DOF generalized igloo Equations~\eqref{eq:iglooconfig1} and \eqref{eq:iglooconfig2} in Theorem~\ref{thm:iglooconfig}.  
\end{example}

\subsection{Two pair, (112234)}\label{sec:2pair}

The $(\rho_1,\rho_1,\rho_2,\rho_2,\rho_3,\rho_4)$ symmetry for $G_{60}$, which we call the \textit{two pair} case, has folding angle relationships that are more difficult to characterize than the previous cases considered.  The vertex-splitting and parallel pleat techniques do not work in this case because the $(112234)$ pattern lacks rotational and reflection symmetry.  That is, splitting the vertex into two flat-foldable degree-4 vertices will not work because there is no way to do this while separating the two creases with folding angle $\rho_1$ \text{and} the two creases with folding angle $\rho_2$; both need to be separated because flat-foldable degree-4 vertices cannot have consecutive creases with equal folding angles while rigidly folding, by Theorem~\ref{thm:deg4}.  
The parallel pleat transform requires the vertex to either have a line of reflection symmetry or certain types of rotation symmetry in order for the folding angle multipliers of Theorem~\ref{thm:deg4} to be consistent across the vertices of the parallel pleat crease pattern.  The symmetry $(112234)$ has neither, and thus this technique will not work for $(112234)$ symmetry.

However, parts of the matrix product in \eqref{eq:R=I} can be used to describe the configuration space for the $(112234)$ case, as we will see in what follows.

\begin{theorem}\label{thm:112234}
The two pair $(\rho_1,\rho_1,\rho_2,\rho_2,\rho_3,\rho_4)$  symmetric rigid folding for $G_{60}$ is a 1-DOF system whose configuration space is described by the following equations:
\begin{equation}\label{eq:2pair5}
\begin{multlined}
(1+3\cos(2\rho_2))\cos\rho_3 + 3(\cos\rho_1-1)\sin\rho_1 \sin\rho_4   \\
=(1+3\cos(2\rho_1))\cos\rho_4 + 3(\cos\rho_2-1)\sin\rho_2 \sin\rho_3
\end{multlined}
\end{equation}
\begin{equation}\label{eq:2pair4}
4\cos\rho_2 + 3\cos(2\rho_2)=
1+2\cos\rho_4-4\sin\rho_1\sin\rho_4+2\cos\rho_1(1+\cos\rho_4)
\end{equation}
\begin{equation}\label{eq:2pair3}
\begin{multlined}  
     24\cos\rho_1+24\cos\rho_2 + 6\cos(2\rho_1) +6\cos(2\rho_2) + 27\cos(2(\rho_1+\rho_2)) - 9\cos(2(\rho_1-\rho_2))  \\
\shoveleft{=24\cos(\rho_1-2\rho_2) + 24\cos(2\rho_1-\rho_2) + 24\cos(\rho_1+\rho_2)+40\cos(\rho_1-\rho_2)-34.}
\end{multlined}  
\end{equation}

\end{theorem}

\begin{proof}
 Let us orient the crease pattern as shown in Figure~\ref{fig:2pairconfig}(a).  We fix the face between creases $c_1$ and $c_6$ and imagine cutting along creases $c_3=\langle -1/2,\sqrt{3}/2,0\rangle$ and $c_4=\langle-1,0\rangle$ to remove one face of the crease pattern.  We pick a value for $\rho_1\in[-\pi,\pi]$ and fold creases $c_1$ and $c_2$ accordingly.  This determines the placement of crease $c_3$, and there will be finitely many values of $\rho_2$ that, when used to fold creases $c_5$ and $c_6$, will place $c_4$ so as to form a $60^\circ$ angle with the placement of $c_3$.  Thus the two pair rigid folding symmetry is a 1-DOF system.
 
Specifically, folding along $c_2$ and then $c_1$ by folding angle $\rho_1$ moves the crease $c_3$ into position
\begin{equation}\label{eq:2pair1}
    R(c_1,\rho_1)R(c_2,\rho_1)c_3 = R_z(\pi/3)R_x(\rho_1)R_z(-\pi/3)R_x(\rho_1)\begin{pmatrix}-1/2\\ \sqrt{3}/2 \\ 0\end{pmatrix}.
\end{equation}
On the other hand, folding along $c_5$ and then $c_6$ by folding angle $\rho_2$ moves the  crease $c_4$ into position
\begin{equation}\label{eq:2pair2}
    R(c_6,\rho_2)R(c_5,\rho_2)c_4 = R_z(-2\pi/3)R_x(-\rho_2)R_z(2\pi/3)R_z(-\pi/2)R_x(-\rho_2)R_z(\pi/3)\begin{pmatrix}-1\\ 0 \\ 0\end{pmatrix}.
\end{equation}
Our crease lines $c_i$ are of unit length, and we want the vectors resulting from \eqref{eq:2pair1} and \eqref{eq:2pair2} to form an angle of $\pi/3$ so that we may insert the face of $G_{60}$ that we cut away. Thus their dot product should equal $\cos\pi/3 = 1/2$.  Simplifying this gives us the relationship between the $\rho_1$ and $\rho_2$ folding angles shown in Equation~\eqref{eq:2pair3}.

Taking a different approach, we can rewrite the matrix product \eqref{eq:R=I} equaling the identity as follows:
\begin{equation}\label{eq:2pair6}
\begin{split}
R_x(\rho_3)R_z(\pi/3)R_x(\rho_2)R_z(\pi/3)R_x(\rho_2)R_z(\pi/3)=\\
R_z(-\pi/3)R_x(-\rho_4)R_z(-\pi/3)R_x(-\rho_1)R_z(-\pi/3)R_x(-\rho_1)
\end{split}
\end{equation}
Simplifying the $(1,1)$ entries of the two matrix products in \eqref{eq:2pair6} gives us Equation~\eqref{eq:2pair4}, while simplifying the $(3,3)$ entries gives us \eqref{eq:2pair5}.  
\end{proof}

\begin{figure}
    \centering
    \includegraphics[width=\linewidth]{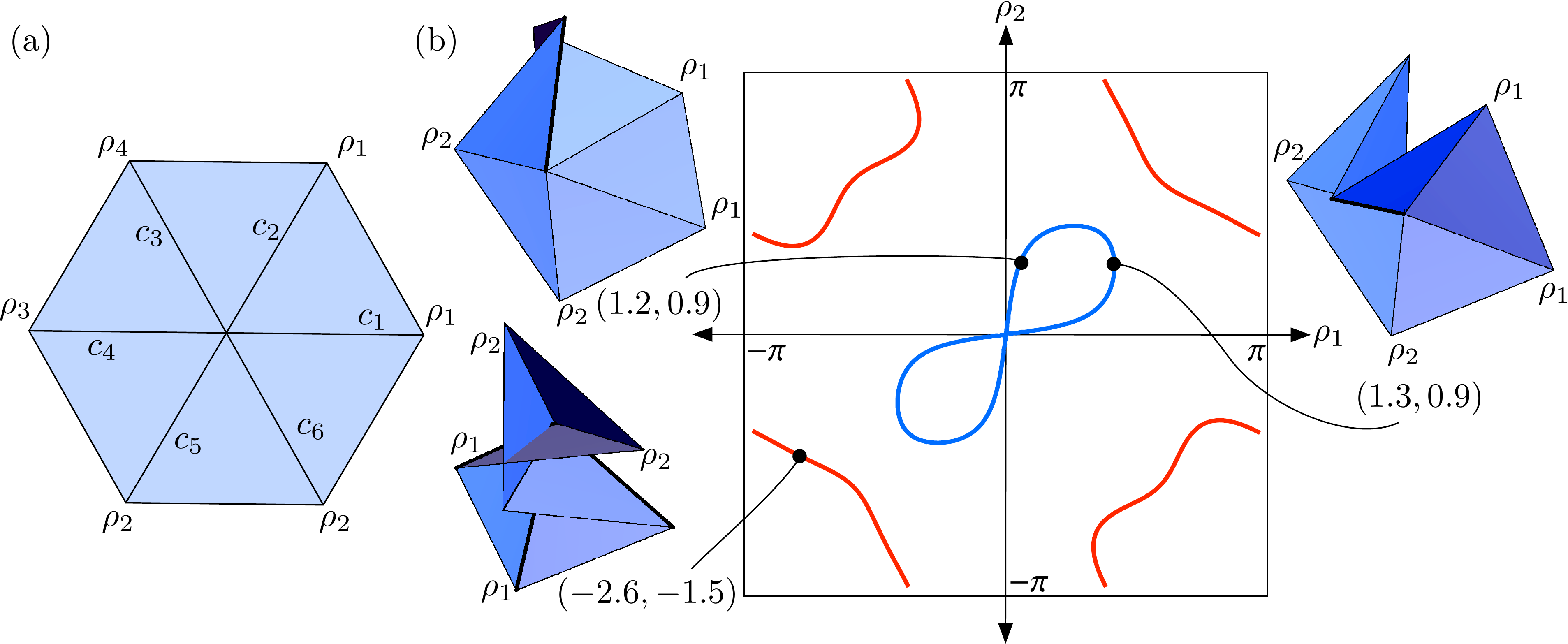}
    \caption{(a) The two pair crease pattern symmetry.  (b) The $(\rho_1,\rho_2)$ slice of the two pair configuration space, in blue, with the rigid folding of two points illustrated. The red curves are solutions to Equation~\eqref{eq:2pair3} that cause the paper to self-intersect, with one example shown.}
    \label{fig:2pairconfig}
\end{figure}

Note that Equation~\eqref{eq:2pair5} illustrates the symmetry of the two pair case; if $(\rho_1,\rho_1,\rho_2,\rho_2,\rho_3,\rho_4)$ is a point in the configuration space, then so is $(\rho_2,\rho_2,\rho_1,\rho_1,\rho_4,\rho_3)$.

The configuration space curve of the valid $(\rho_1,\rho_2)$ values that satisfy  Equation~\eqref{eq:2pair3} forms a figure eight and is shown in Figure~\ref{fig:2pairconfig}(b) in blue, along with images of two different rigid folding examples with the same value of $\rho_2=0.9$ radians.  The curves shown in red also represent values satisfying Equation~\eqref{eq:2pair3} but are not valid because they force the paper to self-intersect.

The equations in Theorem~\ref{thm:112234} defy manipulation attempts to get more simple relations between the folding angles using Weierstrass substitutions.  One reason for this is precisely because, as previously mentioned, we were unable to reduce the $(112234)$ symmetry case for $G_{60}$ to a kinematically-equivalent 1-DOF crease pattern with only degree-4 flat-foldable vertices.

\subsection{Fully and almost general cases}\label{sec:other}

As previously stated, the fully general rigid folding $(\rho_1,\rho_2,\rho_3,\rho_4,\rho_5,\rho_6)$ of $G_{60}$ has no known way to elegantly express its configuration space.  Yet it can be computed using standard kinematic analysis.  For completeness, and as a way to contrast with our previous results for the symmetric cases, we summarize one such method here, based the approach described by Balkcom \cite{Balkcom1,Balkcom2}.

We cut along the crease $c_3=\langle -1/2,\sqrt{3}/2,0\rangle$ to split it into a left side $c_3^l$ and a right side $c_3^r$, as shown in Figure~\ref{fig:fullgen}(a).  The fully general degree-6 vertex rigid folding has 3 DOF, and so we let $\rho_4$, $\rho_5$, and $\rho_6$ be our independent variables and we seek to find folding angle functions for $\rho_1$ and $\rho_2$.  Fix the face between creases $c_1$ and $c_6$.  Then folding along $c_2$ by $\rho_2$ followed by folding along $c_1$ by $\rho_1$ will move $c_3^r$ into position
\begin{equation}\label{eq:gen1}
R(c_1,\rho_1)R(c_2,\rho_2)c_3^r = R_x(\rho_1)R_z(\pi/3)R_x(\rho_2)R_z(-\pi/2)
\begin{pmatrix}-1/2\\ \sqrt{3}/2 \\ 0\end{pmatrix}.
\end{equation}
In the other direction, we fold along $c_4$ by $\rho_4$, then along $c_5$ by $\rho_5$, and finally along $c_6$ by $\rho_6$ to move $c_3^l$ to position
\begin{equation}\label{eq:gen2}
\begin{split}
R(c_6,\rho_6)R(c_5,\rho_5)R(c_4,\rho_4)c_3^l = \\
R_z(-\pi/3)R_x(\rho_6)R_z(\pi/3)R_z(-2\pi/3)R_x(\rho_5)R_z(2\pi/3)R_z(-\pi)R_x(\rho_4)R_z(\pi)\begin{pmatrix}-1/2\\ \sqrt{3}/2 \\ 0\end{pmatrix}.
\end{split}
\end{equation}
The vectors \eqref{eq:gen1} and \eqref{eq:gen2} must be equal, and the $x$-coordinate of \eqref{eq:gen1} is only dependent on $\rho_2$ (since the $R_x(\rho_1)$ matrix does not affect the $x$-axis).  In fact, the $x$-coordinate of \eqref{eq:gen1} is $(1-3\cos\rho_2)/4$, and equating this to the $x$-coordinate of \eqref{eq:gen2} gives us
\begin{equation}\label{eq:gen3}
    \begin{split}
    \cos\rho_2 = \frac{1}{4}(1+\cos\rho_6-2\sin\rho_4 \sin\rho_5 - 2\cos\rho_6\sin\rho_4\sin\rho_5 - 2\sin\rho_5\sin\rho_6+ \\
    \cos\rho_5(1+\cos\rho_6-4\sin\rho_4\sin\rho_6)+\cos\rho_4(1-3\cos\rho_6 +\cos\rho_5(1+\cos\rho_6)-2\sin\rho_5\sin\rho_6)).
    \end{split}
\end{equation}
Since cosine is an even function, Equation~\eqref{eq:gen3} gives us two solutions for $\rho_2$, one positive and one negative.  For $\rho_1$ we could equate the other coordinates of \eqref{eq:gen1} and \eqref{eq:gen2}, but a more robust formula can be obtained by removing the $R_x(\rho_1)$ term from \eqref{eq:gen1}, as the result will be a vector that forms an angle of $\rho_1$ with \eqref{eq:gen2}.  Letting $v=R_x(\pi/3)R_z(\rho_2)R_x(-\pi/3)(-1/2,\sqrt{3}/2,0) = (v_x, v_y, v_z)$ and $u=(u_x,u_y,u_z)$ be \eqref{eq:gen2}, we have
$$v=\langle \frac{1}{4}(1-3\cos\rho_2),\frac{\sqrt{3}}{2}\cos^2\frac{\rho_2}{2},\frac{\sqrt{3}}{2}\sin\rho_2\rangle, \mbox{ and}$$
\begin{equation}\label{eq:gen4}
    \rho_1 = \arctan\frac{u_z}{u_y}-\arctan\frac{v_z}{v_y}= \arctan\frac{u_z}{u_y}-\arctan \left(2\tan\frac{\rho_2}{2}\right).
\end{equation}
(The full expression for $\arctan(u_z/u_y)$ computed from \eqref{eq:gen2} is quite large and uninspiring; it is omitted to save space.) 

\begin{figure}
    \centering
    \includegraphics[width=\linewidth]{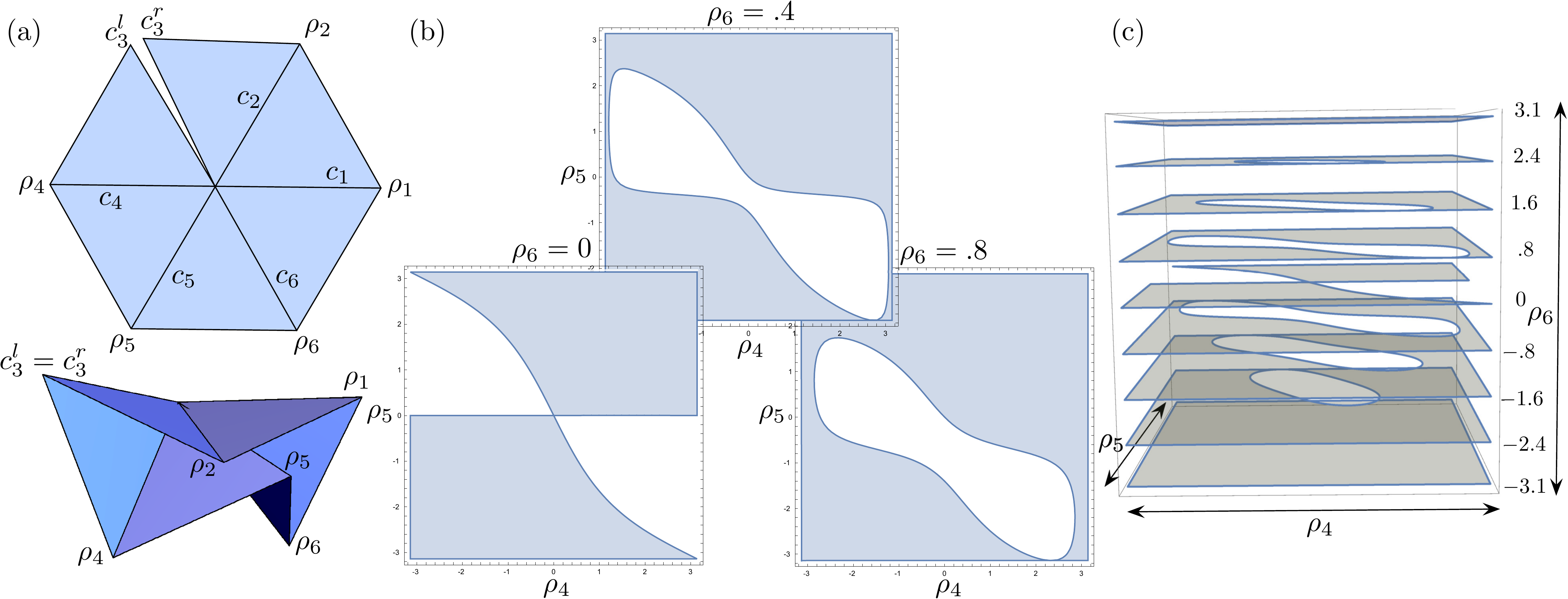}
    \caption{(a) The set-up for analyzing the fully general rigid folding of $G_{60}$, with one rigid folding shown.  (b) Three slices of the $(\rho_4,\rho_5,\rho_6)$ configuration space of ${\cal C}(G_{60})$.  (c) More slices stacked on the $\rho_6$ axis.}
    \label{fig:fullgen}
\end{figure}

Equations~\eqref{eq:gen3} and \eqref{eq:gen4} can be used to simulate general rigid foldings of $G_{60}$, and an example is shown in Figure~\ref{fig:fullgen}(a).  Note that in this 3-DOF system, not all values of $(\rho_4,\rho_5,\rho_6)$ produce solutions for $\rho_1$ and $\rho_2$; e.g. some combinations will force creases $c_1$ and $c_3^l$ to be too far away from each other.  Figure~\ref{fig:fullgen}(b) shows slices of the  $(\rho_4,\rho_5,\rho_6)$ configuration space, for $\rho_6=0$, $.4$, and $.8$,  with the valid $(\rho_4,\rho_5)$ points shaded in.  Figure~\ref{fig:fullgen}(c) shows more such slices stacked on the $\rho_6$ axis to give a sense of the inadmissible region inside this configuration space.

The $(\rho_1,\rho_1,\rho_2,\rho_3,\rho_4,\rho_5)$ rigid folding symmetry of $G_{60}$, which is the $(112345)$ bracelet pattern and is referred to as the ``almost general" case in Table~\ref{table1}, is the least amount of symmetry we could try to impose on the fully general case.  As such, we cannot expect this 2-DOF rigid folding to have elegant folding angle expressions.  However, one may set $\rho_6=\rho_5$ in Equations~\eqref{eq:gen3} and \eqref{eq:gen4} to obtain expressions for use in rigid folding simulations of this case.

\section{Conclusion}\label{sec5}

We have seen that while capturing the kinematic folding angle relationships for rigid foldings of the degree-6 vertex $G_{60}$ crease pattern leads to unwieldy and complicated equations, imposing symmetry on the folding angles yields, in most cases, elegant folding angle expressions, sometimes even linear equations when parameterized with a Weierstrass substitution.  Our proofs have also revealed an explanation for why Weierstrass substitutions work so well in such equations:  In each 1-DOF symmetric rigid folding of $G_{60}$ where Weierstrass substitutions were effective, we were able to reduce the degree-6 crease pattern to a degree-4, flat-foldable (multiple-vertex) crease pattern where the Weierstrass substitution definitely applies by Theorem~\ref{thm:deg4}.  That some higher-degree crease patterns are kinematically-equivalent to degree-4 flat-foldable crease patterns seems a logical reason why Weierstrass substitutions are so effective in rigid origami.

There are many avenues for future work from the results of this paper.  What follows is a list of suggestions and open questions.

\begin{itemize}
    \item Can the folding angle expressions presented in this paper be further simplified?  In the engineering and scientific literature one can find several different (yet equivalent) ways to express the degree-4 flat-foldable folding angle relationships (e.g., see \cite{Stern} and \cite{Izmestiev}), but those shown in Theorem~\ref{thm:deg4} are favorable since they illustrate the tangent half-angle linearity. The degree-6 vertex symmetries presented here that use the Weierstrass substitution are likewise favorable.  But, for example, could the 2-DOF equations for in the generalized igloo (123432) pattern (in Theorem~\ref{thm:iglooconfig}) be improved?
    \item How can we explain the effectiveness of Weierstrass substitutions in 2- or higher-DOF rigid foldings, such as in the opposites (123123) case?  Such cases cannot be reduced to 1-DOF degree-4 crease patterns, so a different reasoning, perhaps one purely algebraic, might lurk within such equations.
    \item In rigid vertex folding symmetries  without rotational or reflection symmetry, reduction to degree-4 flat-foldable crease patterns does not seem possible, as seen in the two pair $(112234)$ pattern. Since the two pair pattern is a 1-DOF rigid folding, it might have simple folding angle relationships yet to be discovered.  Can other techniques be found to handle such cases?
\end{itemize}

We hope that the degree-6 folding angle relationships presented here will offer new tools to mechanical engineers and materials scientists looking for more origami methods to incorporate into their designs.  As Tachi et al. show in \cite{Tachi2}, using the folding angle expressions in Equation~\eqref{eq:trifold} to program spring actuators on the creases of a rigidly folding trifold mechanism results in global convergence of the rigid folding to a desired state.  The new folding angle relationships we have presented should work equally well and provide new rigid folding mechanisms to expand designers' rigid origami vertex repertoire from degree-4 to include degree-6 vertices as well.




\end{document}